\documentclass[11pt,a4paper]{article}
\usepackage[utf8]{inputenc}
\usepackage{amsmath}
\usepackage{amsfonts}
\usepackage{amssymb}
\usepackage{amsthm}
\usepackage{color}
\usepackage{graphicx}
\usepackage{epstopdf}

\theoremstyle{definition}
\newtheorem{definition}{Definition}
\newtheorem{lemma}{Lemma}
\newtheorem{proposition}{Proposition}
\newtheorem{theorem}{Theorem}
\newtheorem{example}{Example}

\newtheorem{remark}{Remark}

\newcommand{\bitem}{\begin{itemize}}
\newcommand{\eitem}{\end{itemize}}
\newcommand{\benum}{\begin{enumerate}}
\newcommand{\eenum}{\end{enumerate}}
\newcommand{\beq}{\begin{equation}}
\newcommand{\eeq}{\end{equation}}

\newcommand{\abs}[1]{\left| #1 \right|}

\newcommand{\norm}[1]{\left\|#1\right\|}

\newcommand{\C}{\mathbb{C}}

\newcommand{\N}{\mathbb{N}}

\newcommand{\R}{\mathbb{R}}

\newcommand{\Z}{\mathbb{Z}}

\newcommand{\calA}{\mathcal{A}}

\newcommand{\calC}{\mathcal{C}}

\newcommand{\calF}{\mathcal{F}}
\newcommand{\calH}{\mathcal{H}}

\newcommand{\calM}{\mathcal{M}}
\newcommand{\calP}{\mathcal{P}}

\newcommand{\calX}{\mathcal{X}}

\newcommand{\abss}[1]{\left\vert #1 \right\vert}
\newcommand{\normm}[1]{\left\Vert #1 \right\Vert}
\newcommand{\set}[1]{\left\lbrace #1\right\rbrace}
\newcommand{\sse}{\subseteq}
\newcommand{\sprod}[1]{\left\langle #1 \right\rangle}

\usepackage{dsfont}
\newcommand{\ind}{\mathds{1}}

\newcommand{\wstarto}{\stackrel{*}{\rightharpoonup}}
\newcommand{\wlim}{\mathrm{w\mbox{-}}\lim}

\DeclareMathOperator{\clonv}{\overline{conv}}

\DeclareMathOperator{\dist}{dist}

\DeclareMathOperator{\sgn}{sgn}
\DeclareMathOperator{\supp}{supp}
\DeclareMathOperator{\id}{id}

\DeclareMathOperator{\ran}{ran}

\DeclareMathOperator{\re}{Re}

\DeclareMathOperator{\sinc}{sinc}

\newcommand{\st}{\text{ subject to }}

 \usepackage[normalem]{ulem}

\newcommand\clsp{\overline{\mathrm{span}}}

\title{Compressed Sensing for Analog Signals}

\author{Bernhard G. Bodmann \and Axel Flinth \and Gitta Kutyniok}

\begin{document}

\maketitle

\begin{abstract}
In this paper we develop a general theory of compressed sensing for analog signals,
in close similarity to prior results for vectors in finite dimensional spaces that are sparse
in a given orthonormal basis.
The signals are modeled by functions in a reproducing kernel Hilbert space.
Sparsity is defined as the minimal number of terms in expansions based on the kernel functions. 
Minimizing this number is  under certain conditions equivalent to minimizing an atomic norm, the pre-dual of the supremum norm
for functions in the Hilbert space. The norm minimizer is shown to exist based on a compactness argument.
Recovery based on minimizing the atomic norm is robust and stable, so it provides controllable accuracy
for recovery when the signal is only approximately sparse and the measurement is corrupted by noise.
 As applications of the theory, we include results on
the recovery of sparse bandlimited functions and functions that have a sparse 
short-time Fourier transform.
\end{abstract}

\section{Introduction}

Sparse regularization techniques are a nowadays  common and successful approach to solve ill-posed inverse problems.
Due to constraints imposed by an application one often faces the problem of having access to only
very few measured quantities, sometimes far less than the ambient dimension of the space in which the solution to the problem resides. The linear version of this problem -- solving an underdetermined linear problem for a solution being sparse in a prescribed
representation system -- has been extensively studied in the area of what is called compressed sensing
or -- seemingly more general, but often used interchangeably -- sparse recovery \cite{DDEK12,candes2006stable,donoho2006compressed}.

The sparse recovery problem  has so far been predominantly treated in discrete settings, following the
tacit assumption that it is inherent in the digital domain only. However, real-world signals are foremost
of analog nature, calling for an adapted algorithmic framework to this structure. Let us exemplarily
consider the problem of determining the dominant frequencies of a musical score. If we assume that the
range of possible frequencies is a discrete set -- e.g., that the instruments are perfectly tuned according to the well-tempered scale and
that the musicians playing the instruments are not making the slightest of mistakes -- we can without
problem convert this into a discrete setting. As soon as the frequencies deviate slightly from the
fixed grid -- i.e., if the instruments are not perfectly tuned -- it becomes impossible to precisely
represent the scores as discrete vectors. This is not merely an academic issue: As an illustrative example,
consider the vector
\begin{align}
	v= (e^{\frac{i\theta k}{m}})_{k=-m}^m \in \C^{2m+1}. \label{eq:v_basismismatch}
\end{align}
Despite its apparent sparsity in the frequency domain, its discrete Fourier transform $\widehat{v}$ has
many nonzero entries as soon as $\theta$ is not an integer multiple of $2\pi$. This phenomenon occurs
much more widely, and is commonly called the \emph{basis mismatch problem} \cite{chi2011sensitivity}.

\vspace*{0.15cm}

In this paper, we will develop a sparse recovery approach suitable for analog signals and provide
precise recovery guarantees, even in the situation of the signals being impacted by noise. The key
ideas of our framework are to consider reproducing kernel Hilbert spaces as ambient spaces with
the kernel functions being the elements of the representation system,  and to solve a total-variation minimization
problem for sparse atomic measures as an extension of classical atomic measures.

\subsection{Desiderata for Analog Signals}

A classical sparse recovery problem aims to recover a vector $x \in \R^n$ from linear, non-adaptive,
and underdetermined measurements $b := Ax$, $A \in \R^{m \times n}$ with $n$ being much larger than $m$.
The key constraint imposed on $x$ is {\em sparsity}, namely that either
\[
\|x\|_0 := \#\{i : x_i \neq 0, i=1,\ldots,n\}
\]
is small or that there exists a representation system $\Phi \in \R^{n \times k}$ -- often coined
{\em dictionary} and assumed to be an orthonormal basis or more generally a frame -- such that
$x = \Phi c$ and $\|c\|_0$  is small. One of the most
popular and well-explored recovery approaches is to solve a convex optimization problem,
more precisely, to find a minimizer for the $\ell_1$-norm that achieves
\[
\min_{\tilde{x} \in \R^n} \|\tilde{x}\|_1 \st A\tilde{x}=b.
\]
Under certain conditions \cite{FoucartRauhut2013}, such a minimizer is indeed the desired solution.

Aiming for a true analog setting, the following four desiderata can be identified:
\bitem
\item[(D1)] {\it Hilbert space model.} A classical mathematical model for analog signals consists of
reproducing kernel Hilbert spaces, one particularly prominent example being the reproducing
kernel Hilbert space of bandlimited functions.
\item[(D2)] {\it General measurements.} In contrast to the classical setting, in which measured quantities
are indexed by a finite or countable, discrete set, we wish to include measurements whose index
set is equipped with a more general metric structure. This includes the case of a measurement whose range is a reproducing kernel Hilbert space.
\item[(D3)] {\it Reconstruction in infinite dimensions.} The reconstruction should
not impose a discretization, but remain in the continuum setting. Solving an appropriate convex optimization problem based on the measurement should allow robust and stable
recovery of the analog signal with respect to the relevant norm.
\item[(D4)] {\it Analog sparsity.} The notion of sparsity itself should not be linked to a specific
discrete set, but in fact -- following the analog spirit -- allow for sparse representations
within an uncountably infinite dictionary.
\eitem

Besides developing an appropriate model setting and recovery strategy, the main goal will consist
in providing precise recovery guarantees. In particular, we will not address the problem of numerically 
solving the infinite dimensional optimization problems we will investigate. It should be noted that in some special cases such as when the linear measurements are
(related to) Fourier samples (see, e.g., \cite{candes2014super,tang2013compressed,dossal2017sampling,heckel2016super}) or
piecewise linear measurement functions \cite{flinthweiss2017exact}, the infinite dimensional problem can in fact be 
solved exactly through finite-dimensional problems. Deriving such statements for our more general situation is beyond the scope of
this paper and will be studied in future work.

\subsection{Related Work}

The discrete paradigm of sparse recovery was challenged in the ground-breaking paper
\cite{chandrasekaran2012convex}, where \emph{atomic norms} with respect to infinite
dictionaries where introduced. Although the concept of atomic norms was certainly  not new
in itself, the authors argued that many of the norms used for structured regularization, i.e.,
the nuclear norm, could be viewed as the atomic norms of certain dictionaries. We will exploit and 
extend the idea of atomic norms in order to realize our envisaged properties (D1)--(D4). One deficiency of
the previously investigated atomic norms is that the norms still relied on discrete representations; in the particular example of
the nuclear norm of a matrix, it is equal to the sum of its finitely many singular values.

A different, and more truly ``non-discrete'' type of  sparse recovery problem
was already treated in \cite{donoho1992superresolution}, before much of the literature on compressed sensing.
The idea in this work was to consider the recovery of a \emph{signed measure} $\mu$ of the form
\begin{align}
	\mu= \sum_{i=1}^s c_i \delta_{x_i} \label{eq:sparseMeasure}
\end{align}
for $c \in \R^s$ and $x_i \in \Omega$, rather than a discretized vector. Hence, the theory does
address (D4). The other desiderata were not considered in that work; the article only treats the problem
of recovering a measure from a part of its Fourier transform $\widehat{\mu}$,
\begin{align} \label{eq:FourierMeasurements}
	\widehat{\mu}(\xi) = \int_{0}^1 \exp(2\pi  i \xi x) d\mu(x),
\end{align}
thus not using a more general type of measurement as in (D2). Moreover, rather than proposing a concrete recovery guarantee,
the article provides a condition on the support of the measure ensuring the \emph{possibility}
of stable recovery of it in form of an oracle inequality, so that (D3) is not fulfilled. The
set-up also does not use explicit signal models as in (D1).

\vspace*{0.15cm}

The analogue of the $\ell_1$-norm in this setting is the total variation norm, or simply $TV$-norm. Having
the success of the $\ell_1$-minimization from the discrete realm in mind, it seems plausible
that, given linear measurements $b= A\mu$, a $TV$-norm minimization of the form
\begin{align} \label{eq:atomNormIntro}
	\min \norm{\mu}_{TV} \st A\mu =b
\end{align}
should be successful at recovering a measure of the form \eqref{eq:sparseMeasure}. Notice
that the model of a measure resolves the basis mismatch problem described before by
viewing $v$ from \eqref{eq:v_basismismatch} as samples of the ``perfectly sparse'' measure
$\delta_{\theta}$.

One of the first publications giving a concrete treatise of this optimization problem,
in particular, giving recovery guarantees, is \cite{de2012exact}. In this work, the authors
consider measurements of the form
\begin{align*}
	\int_\Omega m_k(x) d\mu(x),
\end{align*}
where $(m_k)_{k=1}^m$ is a so-called \emph{M-system}, meaning that, for each $K$, any
function of the form
\begin{align*}
	\sum_{k=1}^K c_k m_k
\end{align*}
possesses at most $K$ zeroes in $\Omega$. Notice that various relevant measurement families
fall into this category such as (trigonometric) polynomials. The authors provide guarantees
for the exact recovery of both positive and more general measures, the latter obeying a
minimal separation condition.

This setting was treated further in the two contributions \cite{candes2014super} and \cite{tang2013compressed}.
The authors of \cite{candes2014super} assumed that all coefficients $(A\mu)_k$
for $\abs{k} \leq m$ are known; whereas in \cite{tang2013compressed} it is assumed that only
a random subsampling is given. In both papers it is then proven that under a condition on the
minimal separation of the $x_i$'s, depending on $m$,  the solution of \eqref{eq:atomNormIntro}
equals \eqref{eq:sparseMeasure}. In comparison to \cite{de2012exact}, the number of measurements needed for resolving measures with a certain separation is much smaller. It should also be mentioned that the authors provide a method for numerically resolving the infinite dimensional $TV$-minimization problem exactly using a finite-dimensional $SDP$.

 Relating again to our list of desiderata, in these three articles neither our model
assumption (D1) is satisfied, nor are  infinite indexed measurement sets allowed.
But within the model setting considered in \cite{candes2014super} and \cite{tang2013compressed},
(D3) and (D4) are indeed fulfilled.

\vspace*{0.15cm}

Since the appearance of \cite{candes2014super} and \cite{tang2013compressed}, research progress
on problems of the form \eqref{eq:atomNormIntro} has in fact been quite limited. A few articles
on purely theoretical considerations, such as existence and structure of the solutions, on
the behavior of discretized versions of \eqref{eq:atomNormIntro} or on questions of identifiability
(e.g., if a measurement operator $A$ is injective on the set of sparse atomic measures) have been
published \cite{unser2017splines, DuvalPeyre2015,flinthweiss2017exact, traonmillin2017, bredies2013inverse}. But
concerning the question of concrete recovery guarantees, the literature is quite scarce.

There are a few notable exceptions. We would like to mention \cite{heckel2016super} and \cite{dossal2017sampling},
in which two-dimensional Fourier measurements are considered. The first paper analyzes the situation of
sampling on a rectangular grid, whereas the second paper studies sampling along radial lines. Although
the articles without doubt contain important and substantial new ideas, the specific Fourier nature of the
measurements allow their arguments to make heavy use of the results from \cite{candes2014super,tang2013compressed}. Hence, they fail to meet the desideratum (D2).

During the final preparation of this manuscript, we also became aware of \cite{Aubel2018}. This paper treats a somehow different type of measurements, namely short time Fourier measurements of measures on $\R$ (or the torus). The authors furthermore allow for the measures to have countably many support points, which by itself imposes many delicate problems not present in the model of finitely many support points considered here. In \cite{Aubel2018}, all of the desiderata are met with the exception of (D2), since the authors only
treat the case of short time Fourier measurements.

\vspace*{0.15cm}

Let us finally comment on another line of work which constitutes a generalization of compressed sensing to infinite dimensions (see, e.g., \cite{adcock2016generalized, eldar2009compressed}). In these papers, it is assumed that the signal $f$ has a sparse representation in a countable representation system $(\psi_k)_{k\in \Z}$, e.g.,
\begin{align*}
	f = \sum_{k\in \Z} d_k \psi_k
\end{align*}
with $d$ being (approximately) sparse (see, for example, \cite{adcock2016generalized,eldar2009compressed}).  This model is quite general,
since it includes cases such as assuming $f$ to be piecewise smooth and choosing $(\psi_k)_{k\in \Z}$ to be equal to a wavelet system, which
then leads to $d$ being approximately sparse. It however cannot capture sparsity in the sense of (D4). 

The main ideas of this line of work is to carefully sample the function $f$ using filters $(s_n)_n$, resulting in a countable compressed sensing problem of the form
\begin{align*}
	\sprod{s_\ell,f } = \sum_{k\in \Z} d_k \sprod{s_\ell,\psi_k}, \ell \in \Z.
\end{align*}
This problem is subsequently, again with great care, approximated with finite dimensional compressed sensing problems. These can be solved with methodologies such as Basis Pursuit, leading to a provably stable recovery procedure. This philosophy is quite different from ours: whereas this approach first discretize and then optimize, we will analyze an infinite dimensional optimization problem directly.  Hence, these
approaches do also not fulfill (D3). We believe that these types of theories should be considered to be entirely separate from, rather than competing with, the one we are treating, since it is designed for and applies to an entirely different class of signals.

\subsection{Our Contribution}

Let us for intuition purposes first focus on the case of Fourier measurements \eqref{eq:FourierMeasurements}.
It is then clear that the recovery of $\mu$ from $M \mu$ is equivalent to the task of recovering the
parameters $(c_i, x_i)$ from samples
\begin{align*}
	\sum_{\ell=1}^s c_i e^{2\pi i k x_i}, \quad k=-m, \dots, m
\end{align*}
of the sum of complex exponentials $P(\xi) = \sum_{\ell=1}^s c_i e^{2\pi i \xi x_i}$. Historically, this
problem was in fact already studied by Prony in the late 19th century. The algorithm he proposed is of
algebraic nature and consists of re-encoding the samples into a polynomial, which then can be proven to have
zeroes exactly in the points $x_i$. In recent years, it has experienced quite a renaissance, for instance, 
in the context of signals with finite rate of innovation
\cite{vetterli2002sampling}.  It is commonly referred to as Prony's algorithm \cite{plonka2014prony}.
This algorithm is however very specifically tailored to the special case of Fourier measurements on a
uniform grid, and is furthermore quite sensitive to noise.

A different interpretation of the problem --~which is of inevitable importance for us~-- is that the sum
of complex exponentials $P$ can be seen as a signal having a sparse representation in the
\emph{dictionary of complex exponentials}
\begin{align*}
	\Phi = \left(\exp(2\pi i x(\cdot))\right)_{x\in [0,1]}.
\end{align*}
With the methodology of \cite{chandrasekaran2012convex} in mind, it seems natural to minimize the
atomic norm over all elements $Q$ (in some appropriately chosen ambient space) having the same --~for now
discrete~-- measurements $(Q(k))_{k=-m}^m$ as $P$, with the goal to find the parameters $(c_i, x_i)$
sparsely representing the element $P$.

In this article, we will follow this general strategy. Our ambient model spaces will be reproducing kernel
Hilbert spaces $\mathcal{H}$, thereby satisfying (D1). The kernel functions $(K_x)_{x \in \mathcal{H}}$ of
those Hilbert spaces will serve as dictionary functions, allowing for an analog version of sparsity (D4).
 The atomic norm with
respect to such dictionaries will be defined through a $TV$-minimization problem. Hence problems of the form
\eqref{eq:atomNormIntro} -- extended to countably infinite indexed measurement sets (D2) -- can readily
be interpreted as atomic norm minimizations. 

This recovery strategy then also fulfills (D3).
Considering this framework for sparse recovery for analog signals, we derive precise recovery guarantees
in the noiseless (Theorem \ref{th:MainRecovery}) and noise-impacted (Theorems \ref{th:Stability} and \ref{th:DeltaMinimize})
situations. The general ideas are the same as the line of work started by \cite{candes2014super, tang2013compressed} -- however, formulating the problem in reproducing kernel Hilbert spaces sheds new light on it, and allows for unification. In particular, the work \cite{candes2014super} can be seen as a special case of the theory. We will also be able to derive results for measurement operators which have not been considered before, related to bandlimited functions and functions with a sparse time-frequency representation.

\subsection{Outline}

This  paper  is  organized  as  follows. In Section \ref{section:atomic norms}, we introduce and prove a few fundamental properties of atomic norms on reproducing kernel Hilbert spaces. Then, in Section \ref{sec:rec}, we present our recovery guarantees. Finally, in Section \ref{sec:appl}, we apply the framework of the previous sections to two concrete settings.

\section{Reproducing Kernel Hilbert Spaces and Atomic Norms} \label{section:atomic norms}

To set up our model situation according to (D1), we start by recalling the key definitions and
notation related to reproducing kernel Hilbert spaces. For the sparse recovery approach, we
in addition require a suitable norm on this space. The so-called atomic norm unified various
sparse recovery settings in Euclidean space, and it seems conceivable that such a type of 
norm will also be highly beneficial for our general setting. This atomic norm was originally 
defined  in \cite{chandrasekaran2012convex} in a slightly different, more geometric manner. 
In the sequel, we will introduce a version specifically adapted to the reproducing kernel 
Hilbert space setting. 

\subsection{Reproducing Kernel Hilbert Spaces Revisited} 

Let us start with recalling the definition of a reproducing kernel Hilbert space.

\begin{definition}
A \textit{reproducing kernel Hilbert space} $\mathcal H$ over a set $\mathcal X$
is a set of functions on $\mathcal X$ forming a Hilbert space such that for each
 $x \in \mathcal X$, there is a  function $K_x \in \mathcal H$ that produces the point
 evaluation of any $f \in \mathcal H$ at $x$ via the inner product with $K_x$, i.e., 
\begin{equation}\label{eq:RKdef}
   f(x) = \langle f, K_x \rangle \, .
\end{equation}
The \textit{reproducing kernel} is the function $K: \mathcal X \times \mathcal X \to \mathbb C$
given by $K(x,y)=K_x(y), x, y \in \mathcal X$. We say that $\mathcal H$ has a \textit{unit-norm 
reproducing kernel}, if $K(x,x) = 1$ for all $x \in \mathcal X$.
\end{definition}

The inner product of $\mathcal H$ induces the usual norm and the topology with respect to which $\mathcal H$ is complete.
From the fact that each $K_x$ furnishes the point evaluation, one can deduce that the span of
$\{K_x\}_{x \in \mathcal X}$ is dense in $\mathcal H$.
If the kernel $K$ is not unit-norm, then one can switch to a normalized
one, $\widetilde K(x,y) = K(x,y)/(K(x,x)K(y,y))^{1/2}$
and pass the weight onto the functions.
This defines a Hilbert space isomorphism $\tilde\iota: \mathcal H \to \widetilde{\mathcal H},
\tilde\iota f(x) = \sqrt{K(x,x)} f(x), x \in \mathcal X$ to the Hilbert space $\widetilde{\mathcal H}$
with unit-norm reproducing kernel $\widetilde K$. This change is useful to show that each function
has a norm-convergent expansion in an at most countable number of kernel functions.
Assuming a unit-norm kernel function, then
the weak greedy algorithm described by Temlyakov \cite{temlyakov} provides 
a norm-convergent  expansion for each $\tilde f \in \widetilde{\mathcal H}$,
and by unitarity, also for each $f \in \mathcal H$ in terms of an at most
countable number of kernel
functions $\{K_x\}_{x \in \mathcal X}$.

 Although each function $f$ can be represented as a countable norm-convergent sequence, 
 it will turn out to be convenient to allow also expansions of the form
\begin{align} \label{eq:Integral}
	K * \mu \equiv \int_\calX K_x d\mu(x),
\end{align}
where $\mu$ is an element of $\calM(\calX)$, the space of regular complex Borel measures 
of bounded total variation on a locally compact Hausdorff space $\calX$.
For the purposes of this paper, we understand the integral of the vector-valued function $x \mapsto K_x$ with respect to the complex measure $\mu$ in the sense of Bochner 
\cite{yosida},
for which it is sufficient to assume that $\mathcal H$ is separable, $K$ is a unit-norm kernel, and each $f \in \calH$ is Borel measurable  \cite{pettis,talagrand}.
A simple condition that ensures measurability  is to have $x \mapsto K_x$ continuous from $\calX$ to $\calH$, which implies through the Cauchy-Schwarz inequality and
the reproducing kernel identity (\ref{eq:RKdef}) that each $f \in \calH$ is continuous.
All of the examples we consider have this property.

\subsection{Atomic Norms on Reproducing Kernel Hilbert Spaces}


The space $\calM(\calX)$ is naturally equipped with the total-variation norm, which assigns to 
$\mu \in \calM(\calX)$ the value
\begin{align*}
	 \norm{\mu}_{TV} \equiv |\mu|(\calX)  = \sup_{ \substack{ \bigcup_{i=1}^n E_i= \calX \\
	E_i \text{ disjoint.}} } \sum_{i=1}^n \abs{\mu(E_i)} \, ,
\end{align*}
where the elements of the partitions of $\calX$  are Borel measurable sets.
The Riesz-Markov-Kakutani theorem identifies this normed space  with the dual of the space $\calC_0(\calX)$ of continuous functions vanishing at infinity. In the following, we will frequently view $\calM(\calX)$ as a locally convex vector space, equipped with the weak-$*$-topology that it obtains as the dual of $\calC_0(\calX)$.

We may now define the atomic norm in $\calH$.

\begin{definition}
Let $\calX$ be a locally compact Hausdorff space and $\calM(\calX)$ the space of regular complex Borel measures with bounded total variation.
Let $\mathcal H$ be a separable reproducing kernel Hilbert space of Borel measurable functions on $\calX$ with a unit-norm reproducing kernel $K$.
The \textit{atomic norm} of $\calH$ associates
with each $f \in \mathcal H$ the value
\begin{align*}
   \|f\|_{A} = \inf \bigl\{ \norm{\mu}_{TV} :  \mu \in \calM(\calX), K*\mu = f  \bigr\}   \tag{$\calP$} \, .
\end{align*}
\end{definition}

By the triangle inequality for the norm, the Radon-Nikodym property \cite{pettis,talagrand},
and the normalization of each $K_x$,
\begin{align*}
\norm{K * \mu} = \norm{\int_\calX K_x d\mu(x)} 
	\le \int_\calX \| K_x \| d\abs{\mu}(x) = |\mu|(\calX) \, ,
\end{align*}
the map $D: \mu \mapsto K * \mu$ is (strongly) continuous. This allows us to prove the following lemma.
\begin{lemma} \label{lem:InfAttain} In addition to the assumptions on $\calX$ and $\calH$, let $\calX$ be second countable and $\calH \cap \calC_0(\calX)$
be dense in $\calH$,
	then for $f \in \calH$ with $\norm{f}_A < \infty$, there exists a $\mu \in \calM(\calX)$ with $f = K * \mu$ and $\norm{\mu}_{TV} = \norm{f}_A$.
\end{lemma}
\begin{proof}
	What we need to prove is that $(\calP)$ has a minimizer. By $\calX$ being second countable, $\calC_0(\calX)$ is separable.
	Take a number $R \in \R$ with $R> \norm{f}_A$.
	Using Helly's version of Banach-Alaoglu for dual spaces of separable Banach spaces, the closed ball $\{ \mu \in \calM(\calX): \|\mu\|_{TV} \le R\}$ is  weak-* sequentially compact. Moreover, $D^{-1}(\{f\})$ is weak-* closed because it is the intersection of the weak-* closed sets
	$M_g = \{\mu: \langle K * \mu, g\rangle = \int_{\calX} \overline{g} d\mu = \langle f, g \rangle\}$
	over all  $g \in \mathcal H \cap \calC_0(\calX)$.
	Hence, any minimizing sequence for $(\calP)$ has a weak-* convergent subsequence. The limit $\mu$ of this subsequence then satisfies
	$K*\mu=f$ and $\|\mu\|_{TV} = \| f \|_A$.
	%
	
\end{proof}

Although the atomic norm does not appear to be discussed in the context of Hilbert spaces with reproducing kernel, its dual is more commonly known.

\begin{proposition} \label{prop:dualNorm}
Given $g \in \mathcal H$, then the dual of the atomic norm
is
$$
  \|g\|_A^* = \sup_{x \in \mathcal X} |g(x)| = \|g\|_\infty \, .
$$
\end{proposition}
\begin{proof}
Using Lemma \ref{lem:InfAttain} and the definition of the dual norm, we see that
\begin{align*}
	\norm{g}_A^* = \sup_{\norm{f}_A \leq 1} \abs{\sprod{g,f}} \leq  \sup_{\norm{\mu}_{TV} \leq 1} \abs{\sprod{g, K*\mu}} = \sup_{\norm{\mu}_{TV} \leq 1}
	| \int g(x) d\mu(x) |  = \norm{g}_\infty \, .
\end{align*}
Hence, $\norm{g}_A^* \leq \norm{g}_\infty$. To see the opposite inequality, let $\epsilon>0$ be arbitrary but fixed and $x$ be such that $\abs{g(x)} \geq (1-\epsilon) \norm{g}_\infty$. We then have
\begin{align*}
	(1-\epsilon) \norm{g}_\infty \leq \abs{\sprod{g,K_x}}\leq \norm{g}_A^*\norm{K_x}_A \leq \norm{g}_A^*,
\end{align*}
since $\norm{K_x}_A \leq 1$ due to $K_x = K * \delta_x$. Since $\epsilon>0$ was arbitrary, this proves $\norm{g}_\infty \leq \norm{g}_{A}^*$.

\end{proof}

Let us now discuss the connection to the original atomic norm definition from 
\cite{chandrasekaran2012convex}. As said before, this notion has a more geometric 
nature: If $\Phi$ is a symmetric subset of a real finite dimensional vector space, 
the authors defined the atomic norm $\norm{\cdot }_{A'}$ as
\begin{align} \label{eq:AltAtomic}
			\norm{x}_{A'} = \inf \set{ t>0 \ \vert \ t x \in \clonv(\Phi)},
		\end{align}
where $\clonv(\Phi)$ is the closed convex hull of the set $\Phi$. When $\calX$ 
is a second countable locally compact metric space, and we take $\Phi = (\omega 
\cdot K_x)_{\abs{\omega}=1, x \in \calX}$, then the previous definition presented
in \eqref{eq:AltAtomic} is equivalent to our $TV$-optimization based definition
also in the infinite dimensional setting. We prepare this statement with an 
approximation lemma that concentrates the support of elements in $\calM(\calX)$.

\begin{lemma}	\label{lem:deltaDenseness}
	Let $\calX$ be a second countable locally compact metric space, then the set of finitely supported complex measures
	\begin{align*}
		\Delta = \set{ \sum_{i \in P} c_i \delta_{x_i} \ \vert \ \abs{P}<\infty, x_i \in \calX, c_i \in \C, i \in P}
	\end{align*}
	is weak-$*$-dense in $\calM(\calX)$. More specifically, for each $ \mu \in \calM(\calX)$, there exists a sequence $(\nu_n)_{n \in \mathbb N}$ in $\Delta$
	such that
	$  \nu_n \wstarto \mu$ and for each $n \in \mathbb N$,  $\norm{\nu_n}_{TV} \leq \norm{\mu}_{TV} .$
\end{lemma}

\begin{proof}
The fact that $\Delta$ is dense in the weak-* topology is by the double-perp theorem equivalent to the pre-annihilator of $\Delta$ being the trivial subspace $\{0\}$.
The Dirac measures are included in $\Delta$, so if $f \in \calC_0(\calX)$ is annihilated by each element in $\Delta$ then it vanishes.

For the approximation result,
let $\mu \in \calM(\calX)$ be given. For each $n\in \mathbb N$, choose a compact set $C_n$ such that $|\mu|(C_n)> \|\mu\|_{TV} - 1/n$.
Next, by the definition of the total variation norm, let $\{A^{(n)}_{k}\}_{k=1}^{m(n)}$ be such that
$
   \sum_{k=1}^{m(n)} |\mu(A^{(n)}_{k})| > |\mu|(C_n) - 1/ n \, .
$
Using a countable basis for the topology consisting of open balls, for any $\delta>0$ we refine the partition to obtain
a Borel measurable partition $\{A^{(n,\delta)}_{k}\}_{k=1}^{m(n,\delta)}$ whose sets
have a diameter of at most
$\delta$.

We now specify the sequence by choosing a sequence of diameters, $\delta_n = 1/ n$; we select a point $x^{(n)}_k \in A^{(n,1/n)}_{k}$ for each $k \le m(n,1/n)$, and let
$$
   \nu_n = \sum_{k=1}^{m(n,n^{-1})} \mu(A^{(n,n^{-1})}_{k}) \delta_{x^{(n)}_k} \, .
$$

By definition, $\|\nu_n \|_{TV} \le \|\mu\|_{TV}$. We wish to prove $\nu_n \wstarto \mu$. For this, let $f \in \calC_0(\calX)$. Since $f$ is vanishing at infinity,
$f$ is uniformly continuous, and hence given $\epsilon>0$, $\abs{f(x) - f(x_k^{(n)})} \leq \epsilon$ for all $x \in A_{k}^{(n,1/n)}$ provided $n$ is large enough.

This implies that
		\begin{align*}
			\abs{\int_{\calX} f d\mu - \int_{\calX} f d\nu_n} 
			&\leq \abs{ \int_{C_n} (f - \sum_{k=1}^{m(n,n^{-1})} f(x_{k}^{(n)}) \ind_{A_{k}^{(n,n^{-1})}} ) d\mu } + \norm{f}_\infty \abs{\mu}(\calX \backslash C_n) \\
			&\leq \norm{\mu}_{TV} \epsilon + \norm{f}_\infty /n  \le C \epsilon
		\end{align*}
		for $n$ large enough, where the constant $C$ only depends on $\mu$ and $f$. Hence, $\nu_n \wstarto \mu$.
		\end{proof}
		
%
		
This now allows us to discuss to which extent both definitions are equivalent.	
		
\begin{proposition} \label{prop:normequiv}
Let $\calX$ be a second countable locally compact metric space, $\calH$ be a separable reproducing kernel Hilbert space over $\calX$ with
unit-norm kernel $K$ such that $\calC_0(\calX) \cap \calH$  is dense in $\calH$.
			Let $\Phi = (\omega \cdot K_x)_{\abs{\omega}=1, x \in \calX}$, then $\norm{\cdot}_{A'}$ defined in \eqref{eq:AltAtomic} is equal to $\norm{\cdot}_A$.
		\end{proposition}

\begin{proof}
Let $f\in \calH$ be arbitrary and $t > \norm{f}_{A'}$. Then there exists a $g\in \clonv \Phi$ such that $f=tg$. By assumption, the function $g$ can be expressed as
a limit $\lim_{n \to \infty} g_n$ with
$$g_n = \sum_{i \in M_n} \theta_i^n \omega_{i}^n K_{x_i^n} = K*\left(\sum_{i \in M_n} \theta_i^n \omega_i^n \delta_{x_i^n}\right),$$ for  sets $\abs{M_n} < \infty$, real, positive, scalars $(\theta_i^n)_{i \in M_n}$ with $\sum_{i \in M_n} \theta_i^n \leq 1$ and unimodular complex numbers $(\omega_i)_{i \in M_n}$.  The corresponding sequence of measures $\mu_n = \sum_{i \in M_n} \theta_i^n \omega_i^n \delta_{x_i^n}$ have norms bounded above by one and therefore, due to sequential weak-$*$ compactness of the closed unit ball in $\calM(\calX)$, there is a weak-$*$-convergent subsequence $\mu_{n'} \wstarto \mu_*$, $\norm{\mu_*}_{TV} \leq 1$. Due to continuity of $D$ as a map from $\calM(\calX)$  to $\cal H$ equipped with the weak topology, $f= tg = \wlim tg_n = \wlim tD \mu_n = tD\mu_*$, which proves that $\norm{f}_{A} \leq \norm{t \mu_*}_{TV} \leq t$. Since $t> \norm{f}_{A'}$ was arbitrary, we conclude $\norm{f}_{A} \leq \norm{f}_{A'}$.
	
	If on the other hand $\norm{f}_A = t$, there exists by Lemma \ref{lem:InfAttain} some $\mu$ with $f = D\mu$ and $\norm{\mu}_{TV}=t$. Lemma \ref{lem:deltaDenseness} shows that there exists a sequence of finitely supported measures $\mu_n = \sum_{i \in M_n} c_i^n \delta_{x_i^n}$ weak-$*$-converging to $\mu$, with $ \norm{\mu_n} \leq \norm{\mu}$ . Continuity of $D$ as above implies that $f = \wlim D \mu_n =  \wlim_{n \to \infty} \sum_{i \in M_n} c_i^n K_{x_i^n}$. By decomposing $c_i^n = t \theta_i^n \omega_i^n$ with $\abs{\omega_i^n} =1$, $\theta_i^n \geq 0$, and $\sum_{i \in M_n}\theta_i^n \leq 1$, we see that $f$ is in the weak
	closure of the convex hull of $t \Phi$, which coincides with the strong closure, so $f \in t\clonv \Phi$. This completes the proof.
\end{proof}

\section{Sparse Recovery by Atomic Norm Minimization} \label{sec:rec}

Having prepared and introduced our general model situation following (D1), we now aim for developing
a theory for sparse recovery based on atomic norm minimization satisfying our desiderata
(D2)--(D4).

\subsection{Sparsity in Reproducing Kernel Hilbert Spaces}

We start by formalizing the notion of sparsity we aim to exploit. Notice that this
definition is indeed in line with (D4). 

\begin{definition}
A function $f \in \mathcal H$ is called {\em $s$-sparse},
if there exist sequences $\{c_n\}_{n=1}^s$ in $\mathbb C^s$ and $\{x_n\}_{n=1}^s$
in $\mathcal X^s$, such that
$$
   f = \sum_{n=1}^s c_n K_{x_n}
$$
and for any other expansion $f = \sum_{n=1}^{s'} c'_n K_{x'_n}$ of $f$ we have
\[
s \le s'.
\]
\end{definition}

It would be useful if any linear combination of $s$ kernel
functions gave an $s$-sparse function. In order to guarantee this,
we need to rule out linear dependencies of the kernel functions. 
This property is present in the function spaces we consider here, but 
proving it can be challenging, as seen in the conjecture by Heil, Ramanathan and Topiwala \cite{HRT}.
The key
notion for ensuring this property is made precise by the following definition. 

\begin{definition}
A Hilbert space with reproducing kernel has the {\em $s$-HRT property},
if for any set of $s$ points $\{x_n\}_{n=1}^s$ in $\mathcal X$
the corresponding kernel functions form a linearly independent
set $\{K_{x_n}\}_{n=1}^s$.
\end{definition}

The next result shows that indeed the $s$-HRT property leads to the required
injectivity when a complex measure with support of size at most $s$ is mapped
the corresponding expansion in terms of kernel functions.

\begin{proposition}\label{prop:uniqueness}
If $\mathcal H$ has the
$2s$-HRT property, then any $f \in \mathcal H$ of the form
$
  f = \sum_{n=1}^s c_n K_{x_n}
$
with $c \in \mathbb C^s$ and $x \in \mathcal X^s$
is $s$-sparse and the choice of $c\in \mathbb C^s$ and $x\in \mathcal X^s$
is uniquely determined by $f$.
\end{proposition}
\begin{proof}
If this were not the case, then there would be two non-identical
linear combinations of at most $s$ terms giving $f$. Equating them
shows that the zero vector can be written as a non-trivial
linear combination of at most $2s$ terms, which contradicts the
$2s$-HRT property.
\end{proof}

The next two examples shall show that the $s$-HRT property is by
far not artificial and in fact occurring in many well-known
situations.

\begin{example} \label{ex:torus}
Let $\mathcal X$ be the torus, identified with $[0,1)$ and the metric distance $d(x,y) = \min\{|x-y|,|x-y+1|,|x-y-1|\}$
between $x, y \in [0,1)$.
For $m \in \mathbb N$, let the Dirichlet kernel be given by $D_m(x) = \frac{1}{2m+1} \sum_{k=-m}^m e^{2\pi i k x}$,
and let $K_y(x) = D_m(x-y)$. Then the space of trigonometric polynomials $\mathcal T_m$ of degree at most $m$, equipped with the
inner product
$$
   \langle p , q \rangle = (2m+1) \int_{0}^1 p(x) \overline{q(x)} dx \,
$$
between $p$ and $q$ in $\mathcal T_m$,
 is a reproducing kernel space.
  By the interpolation property of $\mathcal T_m$, any set of at most $2m+1$ kernel functions
 is linearly independent, hence $\mathcal T_m$ has the
 $(2m+1)$-HRT property.
 \end{example}

 Another example we study further below is a space of entire functions with significance in time-frequency  analysis.
 \begin{example} \label{ex:Bargmann}
 Consider the space $\calF$ containing each entire function $f: \C \to \C$ obtained from a series
 $$
    f(z) = \sum_{n=0}^\infty c_n \frac{z^n}{\sqrt{n!}}
 $$
 with $c \in \ell^2$. The map from $c$ to $f \in \cal F$ becomes an isometry when
the norm of  $f$ is given by an $L^2$-norm with a Gaussian weight,
$$
   \|f \|^2 = \int_{\mathbb C} |f(x+iy)|^2 e^{-(x^2+y^2)} \frac{dxdy}{\pi} \, .
$$
The sequence of monomials $(e_n)_{n=0}^\infty$ with $e_n(z) = z^n / \sqrt{n!}$ is then an orthonormal basis for $\cal F$.
When approximating any entire function $f$ whose weighted $L^2$-norm is finite by $f_m = \sum_{n=m}^\infty \langle f, e_n \rangle e_n$,
its orthogonal projection onto the subspace ${\mathcal F}_m$
of polynomials of degree $m$,
then it can be shown with the mean value property of entire functions that $(f_m)_{m \in \mathbb N}$ converges uniformly on compact sets
in $\mathbb C$, so each such $f$ is in $\calF$. This space is called Bargmann space \cite{Bar67}.

From the resolution of the identity with $(e_n)_{n \in \mathbb N}$ and summability,
$$f(w) = \sum_{n=0}^\infty \langle f, e_n\rangle w^n/\sqrt{n!}$$ which shows that
$
  f(w) = \langle f, K_w \rangle
$
with the reproducing kernel $K_w(z) = \sum_{n=0}^\infty \frac{z^n\overline{w}^n}{n!} = e^{z \overline w}$.
This kernel is not unit norm, but as described we can pass to the space $\widetilde\calF$ with kernel $\widetilde K_w(z) = e^{-(|z|^2+|w|^2)/2}e^{z \overline w}$
whose inner product is given by the Lebesgue measure on $\mathbb C \simeq \mathbb R^2$.
Because of the interpolation property of polynomials, $\calF$ and $\widetilde\calF$ and the $m$-dimensional subspaces
$\calF_m$ and $\widetilde\calF_m$ associated with polynomials of maximal degree $m$ have the $m$-HRT property.
 \end{example}

\subsection{Main Sparse Recovery Result}

In this section, we present our main recovery result. The setting is as follows: We are given 
linear measurements $(\sprod{ f,M_i})_{i\in I}$ of a signal $f \in \calH$, indexed by
an at most countable  set $I$ and corresponding to a family of vectors  $(M_i)_{i \in I}$ in $\calH$,
thereby ensuring satisfaction of (D2). Our standard assumption is that $(M_i)_{i \in I}$ is a Bessel family
with Bessel bound $1$, or equivalently, that the measurement map
$$
  M: \calH \to \ell^2(I), (Mf)_i = \langle f, M_i\rangle
$$
has operator norm $\|M\|=1$. Whenever convenient, we will also replace $\ell^2(I)$ by a general separable Hilbert space $\calH'$ and think of the measurement as a contraction $M: \calH \to \calH'$. Note that the latter interpretation allows for inclusion of even uncountable sampling sets: Observing a function $Mf$ in a reproducible kernel Hilbert space $ \calH'$ on an uncountable domain $\Omega$ amounts to taking uncountably many samples $Mf(x), x\in \Omega)$ of the function $Mf$.

We will assume that
$
	f= \sum_{n=1}^s c_{n} K_{x_n},
$
for some coefficient vector $c \in \C^s$. The first main result of this section is dedicated to finding conditions that ensure
$\norm{c}_1=\norm{f}_A$, that is, the measure $\sum_{n=1}^s c_n \delta_{x_n}$ is a minimizer of the program $(\calP)$. Since it will be very convenient in the sequel, let us introduce a term for such minimizers:
\begin{definition} Let $\calX$ and $\calH$ be as stated in Lemma~\ref{lem:InfAttain},
	and $f \in \calH$ with $\|f\|_A < \infty$. A  minimizer $\mu_*$ of $(\calP)$ is called an \emph{atomic decomposition} of $f$.
\end{definition}
Note that atomic decompositions do not necessarily need to be unique -- but they always exist due to Lemma \ref{lem:InfAttain}. When there is no risk for confusion, we will also very liberally speak of $D\mu_*= \int_{\calX} K_x d\mu_*(x)$  as being an atomic decomposition of $f$.

Inspired by e.g. \cite{chandrasekaran2012convex}, we propose the following program for recovering $f$ from the measurement with $(M_i)_{i \in I}$:
\begin{align}
	\min \norm{g}_A \st \sprod{f,M_i}=\sprod{g,M_i} , i \in I. \tag{$\calP_A$}
\end{align}
We stress that this also following our philosophy for recovering analog signals in the 
sense of (D3).

We now state and prove the first main result on sparse recovery in reproducing kernel Hilbert spaces.
It generalizes results by Cand{\`e}s and Fernandez-Granda \cite{candes2014super} as well as Tang et al.~\cite{tang2013compressed}, based on the use of dual certificates as in \cite{CandesTao2006IEEE}. 

\begin{theorem} \label{th:MainRecovery}
	Let $f \in \calH$ be given by the linear combination 
	\begin{align*}
		f= \sum_{n=1}^s c_n K_{x_n} \,
	\end{align*}
	with $c \in {\mathbb C}^s$ and $x \in \calX^s$.
  Assume that the closed linear span of the measurement vectors $\clsp\set{M_i, i \in I}=\calH'$ contains a continuous function $\psi \in \calH \cap \calC_0(\cal X)$  with the properties
	\begin{itemize}
	\item[(i)] $\norm{\psi}_\infty = 1$,
	\item[(ii)]  $\psi(x_j) = c_j/\abs{c_j}$, $j = 1, \dots, m$,
	\item[(iii)]  $\abs{\psi(x)}<1$ for $x \notin \set{x_j, j=1, \dots m}$,
	\end{itemize}
	 and additionally that $(K_{x_n})_{n=1}^m$ is linearly independent.  Then $f$ is the unique solution of the program
$(\calP_A)$ and $\mu_0 = \sum_{n=1}^s c_n \delta_{x_n}$ is an atomic decomposition of $f$.
\end{theorem}
\begin{proof}
	Towards a contradiction, assume that there exists a $g \in \calH$ with $\sprod{g,M_i}=\sprod{f,M_i}$ for all $i \in I$ but $\norm{g}_A \leq \norm{f}_A \le \|c\|_1$. Let $\mu$ be an atomic decomposition of $g$. By continuity, $\sprod{f,M_i} = \sprod{g,M_i}$ for all $i$ implies that $\sprod{g,\psi}= \sprod{f,\psi}$. This has the consequence
	\begin{align*}
		0=\sprod{f-g,\psi} &= \sprod{K*(\mu_0-\mu), \psi} 
		= \int_{\cal X} \langle K_x ,\psi\rangle d(\mu_0-\mu)
		= \int_\calX \overline{\psi} d\mu_0 - \int_\calX \overline{\psi} d \mu\\
		& =  \sum_{n=1}^m \abs{c_n} - \int_\calX \overline{\psi} d \mu = \|c\|_1 - \int_{\supp \mu} \overline{\psi} d \mu \, .
	\end{align*}
	Now if there were $x \in \supp \mu$ with $\abs{\psi(x)}<1$, there would exist a neighborhood $U$ of $x$ with $\mu(U)>0$ and $\abs{\psi(y)}<1$ for all $y \in U$. This would have the consequence that $\norm{g}_A  \le \|c\|_1 =  \abs{\int_{\supp \mu} \overline{\psi} d \mu}<\norm{\mu}_{TV}$, which is a contradiction to $\mu$ being an atomic decomposition of $g$. Hence, $\supp{\mu} \sse \set{x_n}_{n=1}^m$, and the linear independence property implies that $\mu = \mu^0$, in particular $f=g$.
\end{proof}

The problem addressed in this work is whether
a signal in a space of trigonometric polynomials that satisfies a sparsity pattern requirement can be recovered from projections onto lower-dimensional subspaces.
The subspaces are chosen to consist of low-frequency content so that the measurement observes a low-pass version of the signal, and the recovery restores the signal to its original level of resolution.

\begin{example}
Consider the space of trigonometric functions presented in Example~\ref{ex:torus}.
We wish to recover a sufficiently sparse signal $f \in {\mathcal F}_N$
from a set of measurements with $\{M_j\}_{j=1}^{2m+1}$ with $M_j (t) = e^{2\pi i jt}$ and $m\le N$.

As shown by \cite[Proposition 2.1]{candes2014super}, if $m\ge 128$,
 and $T=\{t_1, t_2, \dots, t_s\}$ is a sequence in $[0,1)$ such that $\Delta(T)=\min_{i \ne j} d(t_i,t_j) \ge 2/m$, then
for any $w \in \mathbb T^s$, there exists $\psi \in  \mathcal F_m$ such that
$\psi(t_j) = w_j$ for each $j \in \{1, 2, \dots, s\}$ and $|h(t)|<1$ if $t \not \in T$.
Consequently, by the preceding theorem, given any $f \in {\mathcal F}_N$
such that
$$
   f = \sum_{n=1}^s c_n K_{t_n}
$$
and $T=\{t_j\}_{j=1}^s$ satisfies the separation condition for $\Delta(T)$,
then $c \in {\mathbb T}^s$, $t \in [0,1)^s$ and $f$ are uniquely
characterized by the projection of  $f$ onto ${\mathcal F}_m$ and by the identity
$$
   \| f\|_A = \sum_{n=1}^s |c_n| \, .
$$
\end{example}

\subsection{Stability and Robustness}

In this section we study the effect of noisy measurements and the question whether accurate recovery can be achieved even if
the model assumption of sparsity is only approximately satisfied.

Let $M$ be the measurement operator on $\calH$, associated with the Bessel family $(M_i)_{i \in I}$ as described in the preceding section.
 Suppose that we are given \emph{contaminated} measurements $b=Mf + n$, of an element $f \in \calH$ with an \emph{approximately} sparse atomic decomposition
 \begin{align*}
 	f= \sum_{n=1}^s c_n K_{x_n} + K* \mu_c,
\end{align*}
in $(K_x)_{x \in \calX}$.  The complex measure $\mu_c$ is hereby arbitrary in $\calM(\calX)$, and intuitively should be thought of as a small `non-sparse' part of $f$.
We wish to find conditions under which the following problem approximately recovers $f$:
\begin{align} \tag{$\calP_A^\epsilon$}
	\min \norm{g}_A \st \norm{Mg-b} \leq \epsilon.
\end{align}

For $\ell_1$-minimization for recovery of sparse signals $v_0 \in \R^n$, theorems like Theorem~\ref{th:MainRecovery} often generalize to statements about approximate recovery by programs like $(\calP_A^\epsilon)$.  A typical statement would be that the solution $x^*$ of a regularized version  $\ell_1$-minimization
	\begin{align*}
		\min \norm{x}_1 \st \norm{Mx-b}_2 < \epsilon
	\end{align*}
	obeys
	\begin{align*}
		\norm{x^*-x_0}_1 \leq C_1 \cdot \left(\min_{c \text{ $s$-sparse.}} \norm{x_0-c}_1\right)+ C_2 \epsilon,
	\end{align*}
for constants $C_1, C_2 \in \R$ (see, for instance, \cite[6.12]{FoucartRauhut2013}). The direct analogue of that result in our setting would be
obtained by replacing the $\ell_1$-norm with the total variation norm, $x^*$ 
 with $\mu_*$, $x_0$ with $\mu_0$ and $c$ with $\mu$, where the minimization is over $\mu$ whose support is of size $s$.
It is unfortunately not possible to prove a statement like this in our setting. To see this, note that often the map $x \mapsto K_x$ is continuous. (In the two cases considered in Section \ref{sec:appl}, this is the case.) Therefore, $\norm{M K_{x'}- M K_{x}}_2$ is small for $x$ close to $x'$. Hence, for certain contaminations $n$, a multiple of $K_{x'}$ will be a solution of $\calP_A^\epsilon$ -- but $\norm{\delta_{x}-\delta_{x'}}_{TV} =2$ for all $x \neq x'$.

In \cite{DuvalPeyre2015}, the author perform a very precise analysis of the structure of the solution of the \emph{Beurling LASSO},
given by
\begin{align*}
	\min_{\mu \in \calM(\mathbb{T})} \frac{1}{2}\norm{\Phi \mu - b}_2^2 + \lambda \norm{\mu}_{TV},
\end{align*}
for the case of $\Phi: \calM(\mathbb{T}) \to L^2([0,2\pi]) $ being a smooth convolution operator, i.e.,
\begin{align*}
	\Phi \mu (t) = \int_{0}^{2\pi} \chi(t-x) d\mu(x) \quad \mbox{with } \chi \in \calC^2(\mathbb{T}).
\end{align*}
$\mathbb{T}$ is thereby the torus $ \mathbb{T} \simeq \mathbb{R}/(2\pi {\mathbb Z})$.  The authors prove that for small noise levels $\epsilon$ and small values of $\lambda$,  under certain dual certificate conditions partly resembling the ones we have discussed above, the solutions $\mu_*$ of the Beurling LASSO are concentrated on sets close to the support of the ground truth solution $\mu_0$. In the particular case of a train of $s$ $\delta$-spikes as ground truth measure, the minimizer of the Beurling Lasso will also be a train of $s$ $\delta$-spikes, whose positions and amplitudes are close to the ones of $\mu_0$.

To prove this is certainly an achievement, but the special form of the measurement operator is explicitly used several times in the arguments. In particular, some of their conditions involve both the first and second derivative of the dual certificate (the function $\psi$ in Theorem \ref{th:MainRecovery}), which does not need to exist for more general measurement operators.

We will prove a result which is of similar flavor, although by far not as precise, as the ones of \cite{DuvalPeyre2015} for $(\calP_A^\epsilon)$. Intuitively, it states that for ground truth signals having an atomic decomposition with small, finite support, most of the mass of the atomic decompositions of the minimizer of $\calP_A^\epsilon$ will be concentrated in a set close to the support of an atomic measure.

We introduce some notation to make the notion of concentration quantitative.
For a sequence $(x_i)_{i \in \mathbb N}$ in the metric space $\calX$ and $\delta >0$, we define $S_\delta= \cup_{i\in \N} U_\delta(x_i)$ where $U_\delta(x)$ is the open ball of radius
$\delta$ centered at $x$. Given a reproducing kernel Hilbert space $\calH$ and a contraction $M: \calH \to \calH'$ to another Hilbert space $\calH'$, for a sequence $(\omega_j)_{j \in \mathbb N}$ and $\lambda,\delta>0$,
we define the following set of interpolating measurements:
    	\begin{align}
		\Theta_{x,\omega,\lambda,\delta} = \{\nu \in \calH': \,  & M^* \nu \in \calH \cap \calC_0(\calX)  \label{eq:consistency}\\
		& M^* \nu (x_j) =\omega_j, j \in \N \label{eq:anchorPoints}\\
		  & \norm{M^* \nu}_\infty \leq 1, \label{eq:bounds}\\
		 & \abs{M^* \nu(x)}\leq \lambda \text{ for } x \notin  S_\delta \} \, .\label{eq:valleys}
	\end{align}


This allows us to derive a stability result of our sparse recovery approach for
analog signals when those signals are affected by noise.

\begin{theorem} \label{th:Stability}

 Let $\calX$ and $\calH$ be as in Proposition~\ref{prop:normequiv} and $M: \calH \to \calH'$ a contraction. Let $f \in \calH$ have the form
	\begin{align*}
		f = \sum_{i \in \N } c_i K_{x_i} + K*\mu_c
\end{align*}	
and let $f_*$ be a solution of the program,\footnote{Such a solution always exists: the set $\set{\mu \vert \norm{Mf-DM\mu} \leq \epsilon, \norm{\mu}_{TV} \leq 2 \norm{f}_A}$ is convex, bounded and closed, thus weak-$*$-compact.} given by
	\begin{align*}
		\min \norm{g}_A \st \norm{b-Mg} \leq \epsilon, \tag{$\calP_\epsilon$}
	\end{align*}
	with $\norm{b-Mf}\leq \epsilon$. Further, define $\mu_*$ through $f_*=D\mu_*$, $\norm{f_*}_A=\norm{\mu_*}_{TV}$,
and for $\delta, \lambda >0$, set 
	\begin{align*}
		 C(\lambda,\delta) &:= \sup_{\abs{\omega_j}=1} \inf  \{\norm{\nu}: \nu \in \Theta_{x,\omega,\lambda,\delta} \}.
	\end{align*}
	Then we have
	\begin{align} \label{eq:measureConcentration}
		\abs{\mu_*}\left(S_\delta\right) \geq \norm{f}_A -  \frac{2C(\lambda,\delta)\epsilon + \norm{\mu_c}_{TV}}{1-\lambda} \, .
	\end{align}

\end{theorem}
\begin{proof}
We first note that $f$ is included in the set over which the atomic norm is minimized, so the minimizer $\mu_*$ satisfies
$
     \|\mu_* \|_{TV} = \|f_*\|_A \le \|f\|_A.
$
Furthermore, combining this with the triangle inequality gives
$$
     \|\mu_* \|_{TV} \le \|f\|_A \le \|c\|_1 + \|\mu_c\|_{TV} \, .
$$
	Denote $\mu_0 = \sum_{i \in \N} c_i \delta_{x_i}$ and pick $\omega_j= \frac{c_j}{\abs{c_j}}$. If there exists no $\nu$ fulfilling the constraint, there is nothing to prove. So let $\nu$ fulfill \eqref{eq:consistency}-\eqref{eq:valleys} and $\psi = M^* \nu$. Due to \eqref{eq:anchorPoints} and the triangle inequality, $\norm{f}_A \le \sprod{K*\mu_0,\psi} + \norm{\mu_c}_{TV}$. Also, $\sprod{K*\mu_0, \psi} = \sprod{K*\mu_0,  M^* \nu} 
	= \sprod{Mf,\nu}$. This implies
	\begin{align*}
		\norm{f}_A \le \sprod{K*\mu_0,\psi }+ \norm{\mu_c}_{TV} &= |\sprod{Mf-Mf_*,\nu}| + |\sprod{K*\mu_*, \psi}| +  \norm{\mu_c}_{TV} \\
		&\leq 2\norm{\nu} \epsilon  + \left|\int_{\calX} \overline{\psi(x)} d\mu_*(x)\right| +  \norm{\mu_c}_{TV},
	\end{align*}
	where the last inequality follows from the constraint of $\calP_\epsilon$, as well as the fact that $\psi \in \calH \cap \calC_0(\calX)$. Now let us take a closer look at the integral. Splitting the domain of integration
	according to
	$$
	   \int_{\calX} \overline\psi d\mu_*= \int_{S_\delta} \overline\psi d\mu_* + \int_{\calX \backslash S_\delta} \overline\psi d\mu_*
	$$
	and estimating $\psi$ by \eqref{eq:bounds} and \eqref{eq:valleys} gives
	\begin{align*}
		\left| \int_{\calX} \overline \psi d\mu_* \right| &\leq \abs{\mu_*}(S_\delta) + \lambda \abs{\mu_*}(\calX \backslash S_\delta) \\
		&\leq   \lambda \norm{f_*}_A + (1-\lambda) \abs{\mu_*}(S_\delta).
	\end{align*}
	Next, we estimate $\norm{f_*}_A \leq \norm{f}_A$. Inserted in the preceding estimate for $\|f\|_A$, this yields
	\begin{align*}
		 \norm{f}_A \leq  2\norm{\nu} \epsilon + \lambda \norm{f}_A + (1-\lambda) \abs{\mu_*}(S_\delta)+ \norm{\mu_c}_{TV} \, .
	\end{align*} Rearranging, we arrive at
	\begin{align*}
		\abs{\mu_*}(S_\delta) \geq \norm{f}_A - \frac{2\norm{\nu} \epsilon + \norm{\mu_c}_{TV}}{1-\lambda}
	\end{align*}
	The definition of $C(\lambda, \delta)$ yields the claim.
\end{proof}

Looking at the proof again, we see that it actually proves that if $\mu$ is a measure satisfying $\norm{\mu}_{TV} \leq \norm{f}_A$ and $\norm{b-MD\mu} \leq \epsilon$,
the existence of the dual certificate $\psi$ implies the estimate  \eqref{eq:measureConcentration}. This fact allows us to prove the following ``interpolation result'' (which is arguably more practically relevant than the above.)
%

\begin{theorem} Let $\calX$, $K$ and $\calH$ be as in Proposition~\ref{prop:normequiv}. \label{th:DeltaMinimize}
	Assume that the map $\calX \to \calH$, $x \mapsto K_x$ is continuous.
	Let
	$$
	   f = \sum_{i=1}^s c_i K_{x_i} + K * \mu_c
	$$
	 and let $(\mu^n) \sse \Delta$ ($\Delta$ still denotes the set of finitely supported measures) be a minimizing sequence of the program
	\begin{align}
		\inf_{\mu \in \Delta} \norm{\mu}_{TV} \st \norm{b-M K * \mu} \leq \epsilon, \tag{$\calP_\Delta$}
	\end{align}
	with $\norm{b-Mf}\leq \epsilon$. Then  there exists a $\delta_0 >0$ so that, for all $\delta\leq\delta_0$ satisfying the same assumptions as in
	Theorem~\ref{th:Stability} and for all sufficiently large $n$, the following error bound holds:
	\begin{align*}
		\norm{D\mu^n -f} &\leq 2 C(\lambda, \delta)\epsilon + (C(\lambda,\delta)+1)  \left( \epsilon \norm{f}_A + \norm{\mu_c}_{TV} + \frac{2 C(\lambda,\delta)\epsilon + \norm{\mu_c}_{TV}}{1-\lambda} \right) \\
		&= C_1 \epsilon + C_2 \norm{\mu_c}_{TV} \, .
	\end{align*}
\end{theorem}

\begin{proof}
	We begin by proving that the optimal value of $(\calP_\Delta)$ is smaller than $\norm{f}_A$: Due to Lemma \ref{lem:deltaDenseness}, there exists a sequence of measures $\upsilon^n_c$ with  $\upsilon^n_c \wstarto \mu_c$ and $\norm{\upsilon^n_c}_{TV} \leq \norm{\mu_c}_{TV}$. The  continuity of $D$ and $M$ then implies that
	 for $\eta \in (0,1)$ sufficiently close to  one,
\begin{align}
	\eta \sum_{k =1}^s c_k \delta_{x_k} + \upsilon_c^n \label{eq:upsilon}
\end{align}
will eventually be feasible for $(\calP_\Delta)$. Hence, by $\eta<1$ the infimum of that program is strictly smaller than $\norm{f}_A$. Since $(\mu^n)$ is a minimizing sequence, $\mu^n$ will therefore eventually satisfy $\norm{\mu^n}_{TV} \leq \norm{f}_A$ and $\norm{b- MD\mu} \leq \epsilon$. Theorem \ref{th:Stability} together with the remark directly prior to the current theorem implies that
	\begin{align*}
		\abs{\mu^\circ}(S_\delta) \geq \norm{f}_A - \frac{2C(\lambda,\delta) \epsilon + \norm{\mu_c}_{TV}}{1-\lambda}=: \norm{f}_A - \epsilon',
	\end{align*}
	where $\mu^\circ$ denotes $\mu^n$ with $n$ large enough. This implies
	\begin{align}
	\abs{\mu^\circ}(\calX \backslash S_\delta) = \norm{\mu^\circ}_{TV} - \abs{\mu^\circ}(S_\delta) \leq \norm{f}_A - ( \norm{f}_A - \epsilon') = \epsilon' \label{eq:ComplementBound}.
	\end{align}
	Next, we define a sequence of complex measures $\overline{\mu}^\circ$ of support $\{x_1, x_2, \dots, x_s\}$ through
	\begin{align*}
		\overline{\mu}^\circ = \sum_{k=1}^s \mu^\circ(U_\delta(x_k)) \delta_{x_k}.
	\end{align*}
	Assuming that $\delta$ is sufficiently small so that  $\{U_\delta(x_k)\}_{k=1}^s$ is a sequence of disjoint sets, we then have
	\begin{align*}
		\normm{\int_{S_\delta} K_x d \mu^\circ(x) - \int_{S_\delta} K_x d\overline{\mu}^\circ(x)} \leq \sum_{i=1}^s \int_{U_\delta(x_i)}\norm{K_{x_i} - K_x} d\mu^\circ  \,
	\end{align*}
	with $S_\delta = \cup_{i=1}^s U_\delta(x_k)$.
	Since $ x \mapsto K_x$ is continuous, there exists a $\delta_0$ such that for all $\delta \leq \delta_0$, $\norm{K_{x_i} - K_x} \leq \epsilon$ for all $i$ and $x \in U_\delta(x_i)$. This, and the fact that $\norm{\nu^n_c}_{TV} \leq \norm{\mu_c}_{TV}$ for all $n$, implies that the above expression can be estimated by
	\begin{align}
		\sum_{i=1}^s \int_{U_\delta(x_i)}\norm{K_{x_i} - K_x} d\mu^\circ \leq \epsilon \mu^\circ(S_\delta)
		.
		\label{eq:MuSqueeze}
	\end{align}
	Inequalities \eqref{eq:ComplementBound} and \eqref{eq:MuSqueeze} together imply
	\begin{align} \label{eq:concentrate}
		\norm{K*(\mu^\circ - \overline{\mu}^\circ)} &\leq \normm{\int_{S_\delta} K_x d \mu^\circ(x) - \int_{S_\delta} K_x d\overline{\mu}^\circ(x)} + \normm{\int_{\calX \backslash S_\delta} K_x d \mu^\circ(x)}  \nonumber \\
		&\leq \epsilon \norm{f}_A + \abs{\mu^\circ}(\calX \backslash S_\delta) \leq \epsilon \norm{f}_A  + \epsilon' \, .
\end{align}

	Next,  let $\psi$ be a function obeying the constraints \eqref{eq:consistency} - \eqref{eq:valleys} with  $\omega_j\in \C$, $\abs{\omega_j}=1$ such that $\overline{\omega_j}
	(\mu^\circ( U_\delta(x_j)) - c_j) = \abs{ \mu^\circ (U_\delta(x_j))-c_j}$.
	We recall that for $\mu^\circ$ feasible,
	$$
	   \|Mf-MK*\mu^\circ\| \le 2 \epsilon
	$$
	so by $Mf = M K*(\mu_s + \mu_c)$, this choice of $\psi$ gives
	\begin{align}
		\norm{\overline{\mu}^\circ-\mu_s}_{TV} &=\sum_{i=1}^s \abs{\mu^\circ(U_\delta(x_j)) - c_j} = \sprod{K*(\overline{\mu}^\circ - \mu_s),\psi}
		= \sprod{MK*(\overline{\mu}^\circ - \mu_s),\nu} \nonumber \\
		&\leq \norm{\nu} (\norm{M K*(\overline{\mu}^\circ - \mu^\circ)}  + \norm{M K*\mu^\circ - M(f-K*\mu_c)}) \nonumber \\
		&\leq \norm{\nu} (\norm{K*(\overline{\mu}^\circ - \mu^\circ)}  + \norm{M K*\mu^\circ - Mf} + \|\mu_c\|_{TV}) \nonumber \\
		&\leq \norm{\nu} (\epsilon \|f\|_A + \epsilon' + 2 \epsilon + \|\mu_c\|_{TV})
		.\label{eq:OnSupport}
	\end{align}
	This finally implies, together with estimate \eqref{eq:concentrate},
	\begin{align*}
		\norm{f - K*{\mu}^\circ} &\leq \norm{f - K*\mu_s} + \norm{K*\mu_s - K*\overline{\mu}^\circ} + \norm{K*(\overline{\mu}^\circ - \mu^\circ)} \\
		&  \leq \norm{\mu_c}_{TV} + \|\overline{\mu}^\circ - \mu_s\|_{TV} + \norm{K*(\overline{\mu}^\circ - \mu^\circ)}\\
		& \leq \norm{\mu_c}_{TV}  + \norm{\nu} (\epsilon \|f\|_A + \epsilon' + 2 \epsilon + \|\mu_c\|) + \epsilon \norm{f}_A + \epsilon' \\
		&\leq 2 \norm{\nu} \epsilon + (\norm{\nu} + 1) (\epsilon \norm{f}_A +\norm{\mu_c}_{TV} + \epsilon' ) .
	\end{align*}
	The definitions of $\epsilon'$ and $C(\lambda, \delta)$ give the claim.
\end{proof}

\section{Applications} \label{sec:appl}

We will now discuss two applications of our methodology, namely sparse recovery in the
situations of bandlimited functions -- as stated in the introduction the ``classical''
example for analog signals -- and of functions with sparse expansions in the short-time Fourier transforms.

\subsection{Sparse recovery for bandlimited functions}

\begin{definition}
The space of bandlimited functions on $\mathbb R$ with bandlimit $1/2$ is also known as the Paley-Wiener space $PW_{1/2}$. It is the space consisting of square integrable functions $f$ whose Fourier transforms $\hat f$ have essential support in $[-1/2, 1/2]$. In particular, they are
obtained from  their Fourier transform by
$$
   f(t) = \int_{-1/2}^{1/2} \hat f(\omega) e^{2\pi i \omega t} d\omega \, .
$$

The norm on $PW_{1/2}$ is the $L^2$-norm.
\end{definition}

From the inclusion $L^2([-1/2,1/2)) \subset L^1([-1/2,1/2))$, it can be inferred that each $f \in PW_{1/2}$ is a continuous function. 
The space $PW_{1/2}$ is indeed a reproducing kernel space on $\calX = \mathbb R$, with reproducing kernel
$$
   K_t(s)  = \int_{-1/2}^{1/2} e^{2\pi i \omega (t-s)} d\omega = \frac{\sin(\pi(t-s))}{\pi(t-s)} \quad \mbox{for  } s \neq t,\; s, t \in \mathbb R,
$$
and $K_t(t)=1$ for $t \in \mathbb R$.

Now assume that a function $f \in PW_{1/2}$ is given through $f = \sum_{j=1}^s c_j K_{t_j} + K*\mu_c$, with $t_j \in [-L/2,L/2]$ and $\mu_c$ a complex Borel measure of bounded total variation.  Intuitively, $f$ is a linear combination of a finitely supported complex measure plus a noise part $\mu_c$, run through a low pass filter; see also Figure \ref{fig:GroundTruth}.
We wish to recover $f$ approximately, assuming the set of points $\{t_j\}$ is sufficiently separated. We note that if $\{t_j\} \subset [-L/2,L/2)$, then the Fourier transform of $f$ satisfies
$$
   \hat f (\omega) =  \sum_{j=1}^s c_j e^{2\pi i t_j \omega} + \hat \mu_c(\omega)
$$
on $\omega \in (-1/2,1/2)$.

Motivated by a result concerning recovery of trigonometric polynomials \cite{candes2014super}, it seems natural to use sampling in the frequency domain as measurements.
However, point evaluations of the Fourier transform are not bounded linear functionals on $PW_{1/2}$. We will circumvent this fact by  instead using ``mollified'' Fourier measurements: Assuming that $(2m+1)/L < 1$, then integrating the Fourier transform~$\hat{f}$ with
a moving window
$$
   \langle f, M_k \rangle = \left( \frac{L}{2\rho} \right)^{1/2} \int_{(k-\rho)/L}^{(k+\rho)/L} \hat f(\omega)d\omega 
$$
for $k\in\{-m, -m+1, \dots, m-1, m\}$, gives perturbed measurements of the Fourier transform of $\mu_0 = \sum_{j=1}^s c_j \delta_{t_j}$
according to
$$
  \langle f, M_k \rangle = \hat f * \hat h_\rho (k/L)
$$
with $h_\rho (x) = \left( \frac{L}{2\rho} \right)^{1/2} \sin(\pi \rho x/L)/(\pi x )$, $\rho \le 1/2$.
We note that with this choice $\{M_k\}_{k=-m}^m$ is an orthonormal system.

\begin{figure} \label{fig:GroundTruth}
\center
 \includegraphics[width=8cm]{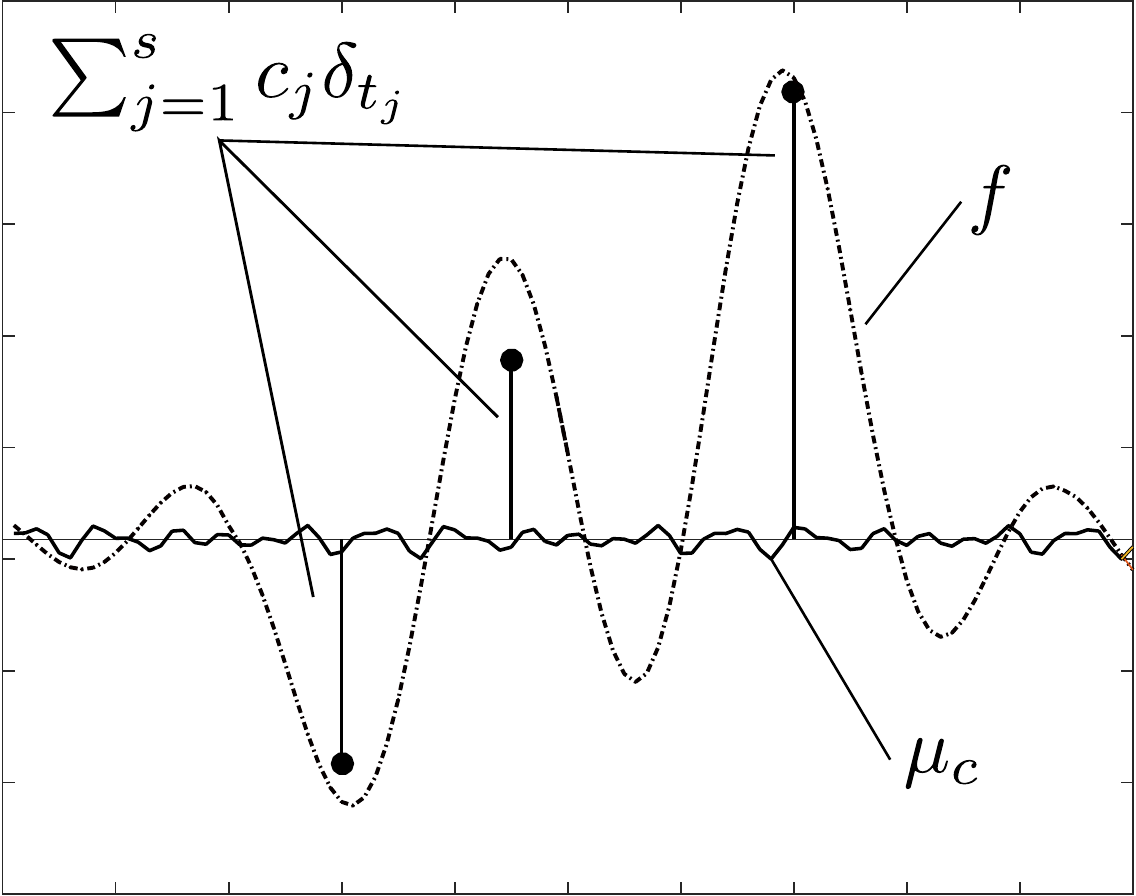}
	\caption{The function $f \in PW_{1/2}$ (dashed) is a low-pass filtered version of the sum of the finitely supported measure $\sum_{i=1}^s c_i \delta_{t_i}$ and a non-sparse part $\mu_c$. }
\end{figure}

It is convenient to re-encode these measured values in a trigonometric polynomial according to
\begin{align*}
 p(y) & = \sum_{k=-m}^m \langle f,M_k \rangle   e^{2\pi i k y}\\
 & =  \sum_{k=-m}^m \hat f * \hat h_\rho (k/L) e^{2\pi i k y} \, \\
%
%
& = \int_{\mathbb R} \sum_{k=-m}^m\exp(2 \pi i k(y-x/L)) f(x) h_\rho(x) dx \\
&= (2m+1) \int_{\mathbb R} D_m(y-x/L) h_\rho(x) d(\mu_0 + \mu_c)(x)
\end{align*}
where we used the notation from Example $1$. The last identity follows from the fact that
$x \mapsto D_m(y-x/L) h_\rho(x)$ is in $PW_{1/2}\cap \calC_0(\mathbb R)$. Hence, $M$  maps functions in $PW_{1/2}$ to the space $\mathcal{T}\C_{\leq m}[x]$ of trigonometric polynomials  on $[-1/2,1/2]$ of degree at most $m$.

In the following, we will investigate to which extent Theorem \ref{th:Stability} is applicable in this context. 
The result of the application of Theorems  \ref{th:Stability} and also \ref{th:DeltaMinimize} to this setting
will be summarized in Theorem \ref{th:Bandlimited}.

 Let us start by noting that the range of the operator $M^*$  is the space $ \{x \mapsto q(x /L) h_\rho(x), q$  a trigonometric polynomial (periodically extended on $\R$) of degree at most $m\}$.
This can be confirmed from the form of $p$ and from the fact that the translates of the Dirichlet kernel $D_m$
span the space of trigonometric polynomials of degree $m$. This fact will enable us to prove the following lemma.

\begin{lemma} \label{lem:AltCert}
	If $\Delta(T) \geq \frac{5}{m}$ and $\rho>0$ is small enough, there exists a trigonometric polynomial $q^\rho$ such that $g = q^\rho h_\rho$ satisfies
	\begin{align}
		g(x_k) &= 1, x_k \in T, \nonumber \\
		\abs{g(x)} &\leq 1- 0.34 \, m^2(x-x_k)^2 \text{ if } \abs{x-x_k} \leq 0.16749/m, \label{eq:nearBound} \\
		\abs{g(x)} &\leq 1- 0.34 \cdot 0.16749^2 \text{ else.} \label{eq:farBound}
	\end{align}
\end{lemma}
\begin{proof}
Considering the form of the range of the measurement operator $M$, it is necessary to construct a trigonometric polynomial $q$ with the properties
\begin{align} \label{eq:newCert}
	q(x_k) &= \omega_k h_\rho(x_k)^{-1}, \\
	q'(x_k) &= - \frac{h'_\rho(x_k)}{h_\rho(x_k)^2},
\end{align}
then the function $q \cdot h$ namely obeys $q(x_k)h(x_k)=\omega_k$, $\partial_x (q(x_k)h(x_k)) = 0$. To do this, we will use an ansatz of the form
\begin{align*}
	q= \sum_{j=1}^s \alpha_j \kappa(\cdot-x_j) + \beta_j \kappa'(\cdot-x_j),
\end{align*}
where $\kappa(x) = \left(\frac{\sin(\pi(m/2+1)x))}{\sin(\pi x)}\right)^4$, just as in the paper \cite{candes2014super}. The interpolation problem then reduces to the following system of linear equations:
\begin{align*}
	\sum_{j=1}^s \alpha_j \kappa(x_k-x_j) + \beta_k \kappa'(x_k-x_j) &= \omega_k h_\rho(x_k)^{-1}, \\
	\sum_{j=1}^s \alpha_j \kappa'(x_k-x_j) + \beta_k \kappa''(x_k-x_j) &= - \frac{h'_\rho(x_k)}{h_\rho(x_k)^2}.
\end{align*}
The proof of Lemma 2.2 in \cite{candes2014super} shows that the matrix
\begin{align*}
	\kappa = \begin{bmatrix} [\kappa(x_k-x_j)]_{k,j=1}^s &  [\kappa'(x_k-x_j)]_{k,j=1}^s \\
	 [\kappa'(x_k-x_j)]_{k,j=1}^s &  [\kappa''(x_k-x_j)]_{k,j=1}^s
	\end{bmatrix}
\end{align*}
is invertible if the minimum separation $\Delta(T)$ obeys $\Delta(T) \geq \tfrac{2}{m}$. Under this assumption, it hence exists for each $\rho >0$ a trigonometric polynomial $q^\rho$ satisfying \eqref{eq:newCert}.

By choosing $\omega_k$ instead of $\omega_k h_\rho(x_k)^{-1}$, and $0$ instead of $ - \frac{h'_\rho(x_k)}{h_\rho(x_k)^2}$, respectively, one obtains a different trigonometric polynomial. Let us call this polynomial $\tilde q$. Assuming $\Delta(T) \geq \frac{5}{m}$, then for $\tilde q$ and for each $x_k \in T$,
\cite[Lemma 2.5]{candes2014super} implies
\begin{align*}
	 \abs{\tilde{q}(x)} \leq 1 - .3354 m^2 (x-x_k)^2 \text{ for } \abs{x-x_k}< 0.16749/m.
\end{align*}
If $\abs{x-x_k} \geq 0.16749/m$ for all $k$, we further have 	$ \abs{\tilde{q}(x)} < 1 - 0.3354 m^2 \cdot 0.16749^2$.

If we prove that $q^\rho h_\rho \to \tilde{q}$ in $\calC^2([-\pi,\pi])$ for $\rho \to 0$, then the inequalities \eqref{eq:nearBound}--\eqref{eq:farBound} will eventually be induced by the corresponding ones for $\tilde{q}$. This is clear for \eqref{eq:farBound}, and as for \eqref{eq:nearBound}, a Taylor expansion around $x_k$ implies
\begin{align*}
\abs{\tilde{q}(x)- q^\rho h_\rho(x)} \leq &\abs{\tilde{q}(x_k)  + \tilde{q}'(x_k) (x-x_k) - q^\rho(x_k) h_\rho(x_k) - \partial_x(q^\rho h_\rho)(x_k)} \\
& + \sup_{y \in [x_k,x]} \abs{ \tilde{q}''(y) - \partial_x(q^\rho h_\rho)(y)} \frac{(x-x_k)^2}{2}  \\
&\leq \norm{\tilde{q}-\partial_x(q^\rho h^\rho)}_{\calC^2}  \frac{(x-x_k)^2}{2}
\end{align*}

First, we have $h_\rho(x) = \sqrt{\rho/2L} \sinc(\rho x/L)$. Consequently, $h_\rho'(x) = \sqrt{\rho/2L} \cdot \rho/L \pi \sinc'(\rho x/L)$ and $h_\rho''(x) = \sqrt{\rho/2L} \cdot \rho^2/L^2 \pi^2 \sinc''(\rho x/L)$. Therefore,
\begin{align*}
	\sqrt{2L/\rho} h_\rho &\to 1 \\
	\sqrt{2L/\rho} \cdot L/\rho h_\rho'(x)& \to 0 \\
	\sqrt{2L/\rho} \cdot L^2/\rho^2 h_\rho''(x) &\to  \frac{\pi^2}{6},
\end{align*}
where the convergence is uniform on any compact interval.  This actually proves that $\sqrt{2L/\rho} h_\rho \to 1$ in $\calC^2([-R,R])$ -- the extra $\rho$'s in front prevents the second derivative to converge to zero -- removing them makes it do that.


Now we notice that the trigonometric polynomial $q_\Delta = (\rho/2L)^{1/2} q^\rho - \tilde{q}$ solves the interpolation problem
\begin{align*}
	q_\Delta(x_k) &= \omega_k ( (\rho/2L)^{1/2}h_\rho^{-1}(x_k) - 1) \\
	q_\Delta(x_k) &= - \frac{\rho^{1/2} h'_\rho(x_k)}{(2L)^{1/2} h_\rho(x_k)^2}.
\end{align*}
Due to construction, $q_\Delta$ is linearly dependent on the finitely many values $w=(q_\Delta(x_k), q'_\Delta(x_k))$. Hence, if we prove that $w \to 0 $ as $\rho \to 0$, $q_\Delta \to 0$. The latter convergence is furthermore true in any norm on $\mathcal{T}\C_{\leq m}[x]$, in particular the $\calC^2[-1,1]$-norm, since all norms on  finite dimensional spaces  are equivalent. But  $((\rho/2L)^{1/2}h_\rho^{-1}(x_k) - 1)\to 0$ is a consequence of $(2L/\rho)^{1/2}h_\rho \to 1$, and
\begin{align*}
	\lim_{\rho \to 0} \sqrt{\frac{\rho}{2L}} \frac{h_\rho'(x_k)}{h_\rho^2(x_k)} = \lim_{\rho \to 0}\sqrt{\frac{\rho}{2L}} \frac{\sqrt{\rho/2L}\pi\rho/L \sinc'(\rho x_k/L)}{\rho/2L \cdot \sinc^2(\rho x_k/L)} = \lim_{\rho \to 0} \frac{\pi \rho}{L}\frac{ \sinc'(\rho x_k/L)}{ \sinc^2(\rho x_k/L)} =0.
\end{align*}
We conclude that $(\rho/2L)^{1/2} q^\rho \to \tilde{q}$ and  $(2L/\rho)^{1/2} h_\rho \to 1$ in $\calC^2([-1,1])$. Consequently, $q^\rho h_\rho \to \tilde{q}$ in the same sense, and the claim has been proven.
%
%
%
%
%

\end{proof}



We now assume that we observe a noisy version of $Mf$, i.e. $Mf +\upsilon$, $\norm{\upsilon}_2\leq \epsilon$ (this is in particular the case if we assume that $\upsilon(y) = \sum_{k=-m}^m \lambda_k \exp(ik\pi y)$ with $\norm{\lambda}^2 \leq  \epsilon$, i.e. that we build $Mf$ via noisy Fourier samples $\sprod{f, \exp(i k \pi(\cdot))} + \lambda_k$. Using the previous lemma, we can derive the following result. 

\begin{theorem} \label{th:Bandlimited}
 	Assume that $\Delta(T)\geq \frac{5}{m}$. Then there exists a $\rho_0>0$ so that for every $\rho\leq \rho_0$,
the following is true for the solution $\mu_*$ of $\calP_{\calA}$
\begin{align} \label{eq:BandLimitStab}
\abs{\mu_*}\left(S_{\min(\delta, .16749/m)}\right) \geq \norm{f}_A -  \frac{2\sqrt{\frac{2L}{\rho \sinc(\rho/2L)^2}} \epsilon + \norm{\mu_c}_{TV}}{1-.34 \cdot \min(\delta, .16749)^2},
\end{align}
where $S_\delta$ is as in Theorem \ref{th:Stability}.
Furthermore, a minimizing sequence $(f_n)$ for the $TV$-minimization problem of the form $\sum_{t \in \Delta_n} c_t K_t$ will eventually satisfy
\begin{align*}
	\norm{f_n -f} \leq & \sqrt{\frac{L}{2\rho \sinc(\rho/2L)^2}} \epsilon \\
	&+ \left(\sqrt{\frac{2L}{\rho \sinc(\rho/2L)^2}}+1\right) \left(\norm{f}_A \epsilon + \norm{\mu_c}_{TV} + \frac{\sqrt{\frac{2L}{\rho \sinc(\rho/2L)^2}}\epsilon + \norm{\mu_c}_{TV}}{1-\min(\delta, .16749)^2} \right).
\end{align*} 	

\end{theorem}

	\begin{proof}
	By construction, $\psi= q^\rho(\cdot /L)h_\rho$, where $q$ is the trigonometric polynomial constructed in Lemma \ref{lem:AltCert} can be used as $\psi$ in the stability theorem. Since for $\nu \in \mathcal{T}\C_{\leq m}[x]$
	\begin{align*}
		M^* f(x)= h_\rho(x) \int_{-1/2}^{1/2} D_m(y - x/L) \nu(y) dy = h_\rho(x) \nu(x/L),
	\end{align*}
we see that the $\nu$-function in the stability theorem can be identified with  $q^\rho$.
	
	Since $\norm{q^\rho h_\rho}_\infty \leq 1$, we can conclude	
	\begin{align*}
		\norm{\nu}_2^2= \norm{q^\rho}_2^2 \leq \int_{-1/2}^{1/2} h_\rho^{-2}(x) dx = \frac{2L}{\rho} \int_{-1/2}^{1/2} \sinc(\rho x/L)^{-2} dx \leq \frac{2L}{\rho \sinc(\rho/2L)^2}
	\end{align*}
	
		Lemma \ref{lem:AltCert} (in particular the bounds \eqref{eq:nearBound} and \eqref{eq:farBound}) now implies that the constant $C(\lambda, \delta)$ in Theorem \ref{th:Stability} satisfies
 	\begin{align*}
		C( 0.34 \cdot 0.16749^2, L \cdot 0.16749/m) \leq \sqrt{\frac{2L}{ \rho\sinc(\rho/2L)^2}} ,&  \\
		C( 0.34 \cdot m^2 \delta^2, L \cdot 0.16749/m) \leq \sqrt{\frac{2L}{\rho \sinc(\rho/2L)^2}},& \quad \delta \leq 0.3298/m.
 	\end{align*}
		
	 (Note that Lemma \ref{lem:AltCert} is to be applied for the points $x_k/L$!) Applying Theorems \ref{th:Stability} and \ref{th:DeltaMinimize} now immediately yields the statements of the theorem.\end{proof}

\begin{remark}Note that the constant in front of $\epsilon$ in \eqref{eq:BandLimitStab} tends to infinity as $\rho \to 0$. On the other hand, $2/(1-0.34\cdot 0.16749^2 ) \leq 1.00963$, so the constant in front of $\norm{\mu_C}_{TV}$ is quite moderate.

\end{remark}

\subsection{Gaussian Window Radar Measurements.}


Next, we consider another model for sparsity, which is motivated by an application to Doppler radar.
Very rudimentarily explained, a radar device functions as follows: A waveform $h$, say in $L^2(\R)$, is sent out from the radar antenna. This waveform is reflected by different objects, and is then measured by a receiving antenna as it returns. The transmission and reflection cause the signal to lose power and shift phase, and since the reflection happens at a distance from the device, it returns with a delay. In addition, if the reflecting object is moving, the received signal is shifted in frequency. Assuming that the reflection takes place at distinct point-like objects, the signal $r$ acquired by the receiving antenna is a linear combination of time-frequency shifted
copies of the signal $h$,
\begin{align*}
 	r(t)= \sum_{i=k}^s r_k h(t-\tau_k) e^{2 \pi i\omega_k t}.
\end{align*}
$r_k \in \R$ are reflection coefficients modelling the loss of power, $\tau_k \in \R$ are the delay and $\omega_k$ the frequency shifts.

This problem was treated in \cite{heckel2016super}. By assuming that $h$ is a trigonometric polynomial, the authors were able to rewrite the problem as a super resolution task from 2D Fourier measurements. This strategy allow the authors to apply techniques from \cite{candes2014super,tang2013compressed} to come up with a recovery guarantee.

  Here, we will instead assume that the \emph{original waveform is a  Gaussian}, i.e. $h(t) = e^{-\Lambda t^2}$, $\Lambda>0$, and hence fully leave the Fourier regime. This filter is certainly not non-realistic,  but it should be noted that its exact form will be essential for our arguments. The received signal can then be rewritten as follows:
\begin{align*}
	r(t) = \sum_{k=1}^s r_k e^{-\Lambda(t-\tau_k)^2} e^{2 \pi i\omega_k t}
	\end{align*}
The problem at hand is now to reconstruct the parameters $r_k$, $\tau_k$ and $\omega_k$ from the knowledge of the function $r$. In the following, we will show that we can map $r$ onto  a signal $\tilde{r}$ which is a short linear combination of the kernel functions, indexed by $(\tau_k,w_k)$, of the so called \emph{Bargmann space}. This will allow us to apply the framework we have developed to recover $(\tau_k,w_k)$ by atomic norm minimization. Towards this goal, let us begin to describe the route from the radar measurements $r$ to the signal in the Bargmann space $\tilde{r}$.

%

\begin{definition}
Let $\{T_a\}_{a\in \mathbb R}$ be the family of translation operators,
acting on $f\in L^2(\mathbb R)$ by $T_af(x)=f(x-a)$  for Lebesgue-almost every $x \in \mathbb R$.
The \emph{short-time Fourier transform} of a function
$f \in L^2(\mathbb R)$ with a window $g\in L^2(\mathbb R)$ is given by the function $V_gf$,
defined on the time-frequency plane $\mathbb R^2$. The transform $V_g f$ has
 the value
$$
  V_g f (t,\omega) = (fT_t \overline g){\hat{~}} (\omega)= \int_{\mathbb R} \overline{g(s-t)} e^{- 2 \pi  i \omega s} f(s) ds\,,  \qquad (t,\omega) \in \mathbb R^2.
$$
\end{definition}

Henceforth, we exclusively use $L^2$-normalized Gaussian windows, $g_\Lambda(s)= (\frac{2}{\pi})^{1/4} \Lambda^{1/2} e^{-\Lambda^2 s^2}$, and replace the
subscript $g_\Lambda$ with the parameter $\Lambda>0$,
$$
    V_\Lambda f(t,\omega) = (\frac{2}{\pi})^{1/4} \Lambda^{1/2} \int_{\mathbb R} e^{-\Lambda^2 (s-t)^2 - 2 \pi i \omega s} f(s) ds\, , \qquad (t,\omega) \in \mathbb R^2 .
$$
Using the isometry of the Fourier transform together with the normalization assumption on $g$, one easily proves that $V_\Lambda$ is an isometry and consequently maps $L^2(\mathbb R)$ to a closed subspace of $L^2({\mathbb R}^2)$. This space is a reproducing kernel Hilbert space, as the following lemma shows.
%
%
%

\begin{lemma} \label{lem:RangeVLambda}
 The range of $V_\Lambda$ is a reproducing kernel Hilbert space with a unit-norm reproducing kernel given by
$$
  K_{t,\omega}(t',\omega') = e^{-\Lambda^2(t-t')^2/2  -\pi^2(\omega-\omega')^2/2 \Lambda^2} e^{i \pi (t+t')(\omega-\omega')} \, .
$$
\end{lemma}
\begin{proof}
By definition of $V_\Lambda$,
$
  V_\Lambda f(t,\omega) = \langle f, M_\omega T_t g_\Lambda\rangle
$
hence with the isometry property,
$K_{t,\omega} = V_\Lambda M_\omega T_t g_\Lambda$ is the claimed reproducing kernel. Moreover,
by the normalization of $g_\Lambda$, it is a unit-norm kernel. The explicit expression for the kernel follows
from $K_{t,\omega}(t',\omega') = \langle M_\omega T_t g_\Lambda, M_{\omega'} T_{t'} g_\Lambda \rangle$.
\end{proof}

We also define a transform $A_\Lambda$, which maps $f$ to an analytic function, by
$$
   A_\Lambda f(z) = (\frac{2}{\pi})^{1/4} \Lambda^{1/2} e^{z^2/2} \int_{\mathbb R} e^{ -(z-  \Lambda s)^2} f(s) ds \, , z \in \mathbb C .
$$
To establish the relation between $V_\Lambda$ and $A_\Lambda$,  we map a point
$(t,\omega)$ in the time-frequency plane to $z=\Lambda t - i \pi \omega/\Lambda$ and verify
$$
  A_\Lambda f(z) =
  e^{i z_1 z_2} e^{(z_1^2+z_2^2)/2} V_\Lambda f(z_1/\Lambda,z_2\Lambda/\pi) \, .
$$
Hence, we can interpret $A_\Lambda$ as a short-time Fourier transform in complex coordinates with a change of measure.
The range of this transform is Bargmann space \cite{Bar61,Bar67}  from Example~\ref{ex:Bargmann}.

\begin{lemma}\label{lem:RangeALambda}
The map $A_\Lambda$ is an isometry with the space $\mathcal F$ as its range.
\end{lemma}
\begin{proof}
The map $A_\Lambda$ is seen to be an isometry from the relationship between $A_\Lambda$ and $V_\Lambda$, the change of measure $\frac{dz_1 dz_2}{\pi} =dtd\omega$
and from the fact that $V_\Lambda$ is an isometry. The fact that the space $\mathcal F$ is the range of $A_\Lambda$ follows from the fact that each monomial
is in the range of $A_\Lambda$. Monomials can be obtained by repeatedly differentiating the kernel function $K_w$ with respect to $w$ and then setting $w=0$.
This is admissible because the vector-valued function $w \mapsto K_w = \sum_{n=0}^\infty \frac{\overline{w}^n}{\sqrt{n!}} e_n $ is arbitrarily often differentiable.
\end{proof}

\begin{lemma}
If  $h=g_\Lambda$ and $r$ is a finite linear combination as described, then $A_\Lambda r$
is a linear combination of kernel functions with indices located at $z_k = \Lambda\tau_k - i\pi \omega_k/\Lambda$.
\end{lemma}
\begin{proof}
On the space $\mathcal F$, the time-frequency translations act
unitarily by
$$
  A_\Lambda T_t f(z) = e^{-\frac \Lambda 2 t ^2} e^{t z} f(z-\Lambda t) \,
$$
and
$$
  A_\Lambda M_\omega f(z) = e^{-\frac{\pi^2}{2\Lambda^2} \omega ^2} e^{i \pi \omega z/\Lambda} f(z-i \pi \omega/\Lambda) \,
$$
\end{proof}




By Lemma \ref{lem:RangeVLambda},  $V_\Lambda r$ is an element of the reproducing kernel Hilbert space
$$
  \mathcal H = \{g: g(z)= e^{-|z|^2/2} f(z), f \in \mathcal F \} \,
$$
with unit-norm reproducing kernel
$
  K(z,w) = e^{-(|z|^2+|w|^2)/2} e^{z \overline w}.
$ Hence, by choosing the measurement operator $M$ equal to the identity, we are in the setting of our main results with $\calH'=\calH$. 
We write $\{\eta_w\}_{w \in \mathbb C}$ for the normalized family of kernel functions,
$$
   \eta_w(z) = e^{- \frac 1 2 |w|^2} e^{\overline w z} \, ,
$$
so
$
   \langle f, \eta_w \rangle = e^{- \frac 1 2 |w|^2} f(w) \, .
$

\begin{remark}
	Note that although it may seem that our setting is related to \cite{Aubel2018}, the two problems are in fact quite different: In \cite{Aubel2018}, the problem is to reconstruct a measure defined on $\R$ from short time Fourier measurements. We on the other hand are essentially assuming that we are given the \emph{adjoint} of the short time Fourier transform applied to a measure defined on $\R^2$.
\end{remark}

Let us remark that $\mathcal{H}$ is more regular than a standard reproducing kernel Hilbert space, in the sense that not only the evaluation of a function itself, but also its derivatives, is continuous on $\calH$. The following proposition is a consequence of standard calculations:
\begin{lemma} \label{lem:DerivativeVectors}
\begin{itemize}
\item[(i)]
For $w, z \in \C$, consider the function
	\begin{align*}
		d\eta_w(z) = e^{-\frac{\abs{w}^2+\abs{z}^2}{2}} e^{\overline{w}z}(z-w) \in \mathcal{H}
		\end{align*}
		Then for $g \in \calH$, we have
		\begin{align*}
			\sprod{d\eta_w,g} =g'(w)
		\end{align*}

		\item[(ii)] Define the function $d^2\eta_w \in \mathcal{H}$ through
		\begin{align*}
			d^2\eta_w(z) = e^{-\frac{\abs{w}^2+\abs{z}^2}{2}} e^{\overline{w}z}(z-w)^2
		\end{align*}
		Then, for $g \in \mathcal{H}$, the Hessian $H_w$ (in Wirtinger notation-- see for instance\footnote{The notation used here differs slightly from the form given in \cite{kreutz2009complex}, since the complex gradient is defined slightly different there. We choose this form to not avoid unnecessary complications.}) \cite{kreutz2009complex} of the function $z \mapsto \abs{g(z)}^2$ is given by
		\begin{align*}
			H_w(\abs{g}^2) = \begin{bmatrix}
				\abs{\sprod{d\eta_w, g}}^2 - \abs{\sprod{K_w,g}}^2 & \overline{\sprod{K_w,g}}\sprod{d^2\eta_w,g} \\ \overline{\sprod{d^2\eta_w,g}}\sprod{K_w,g} & \abs{\sprod{d\eta_w, g}}^2 - \abs{\sprod{K_w,g}}^2
			\end{bmatrix}
		\end{align*}
		\end{itemize}
\end{lemma}

Now we are in a position to prove the main theorem of this section.  The technique will be similar to above: towards applying Theorem \ref{th:MainRecovery}, we will construct a function  $g \in \ran V_\Lambda = \calH$   with $g(w_k) = u_k$ and $\abs{g(w)}<1$ for $w \in \C \backslash \set{w_k, k=1, \dots s}$, for unimodular values $(u_k)_{k=1}^s$. If such a function exists, it certainly satisfies $\frac{d}{dw}\abs{g(w_k)}^2=0$ for $k= 1, \dots s$. Our strategy will therefore be to try and find a function $g \in \calH$ which satisfies
\begin{align}
	\begin{cases} g(w_k) &= u_k\\
	\frac{d}{dw} \abs{g(w_k)}^2 &=0\end{cases}, \quad  k =1, \dots s \label{eq:interpolate}
\end{align}
and subsequently check under which conditions this function really obeys $\norm{g}_\infty =1$. (Compare \cite{candes2014super} and \cite[Section 4]{DuvalPeyre2015}.)  The idea of the proof will be similar to the mentioned sources, i.e.  to make an ansatz, conclude that the conditions \eqref{eq:interpolate} is equivalent to a system of linear equations, prove that that is solvable, and then finally check that $\abs{g} \leq 1$ for the resulting function. Under a few conditions on the set $W=\set{w_k \, \vert, \, k=1, \dots, s}$, this strategy will be successful.

 The formal theorem and proof are as follows:

\begin{theorem}
	Let  $R\geq 1$, $\Omega = B_R(0) \sse \C$, $W= \set{w_k \, \vert \, k=1, \dots s} \sse \C$ be a set of time-frequency shifts and $f_0 \in \calF$ a signal of the form
	\begin{align*}
		f_0= \sum_{w_k \in W}  c_{k}K_{w_k}.
\end{align*}
	 Define the quantities
	\begin{align*}
		&\overline{\sigma}(W,z) = \sum_{j=1}^s e^{-\frac{1}{2}\abs{z-w_j}^2} ( 1+ \abs{z-w_j} + \abs{z-z_j}^2+ \abs{z-w_j}^3) \\
		&\overline{\sigma}_{\downarrow}(W) = \sup_{i} \overline{\sigma}(W,w_i), \quad
				\overline{\sigma}_{\uparrow}(W) = \sup_{z \in \C} \overline{\sigma}(W,z).
	\end{align*}
Suppose that there exists a $\tau>0$ for which
\begin{align}
	\sup \set{ \sum_{j=1}^s e^{-\frac{1}{2}\abs{z-w_j}^2} \, \vert \,  \dist(z,W)>\frac{\sqrt{5}-1}{4\overline{\sigma}(W)_{\uparrow}}}  < 1- \tau \label{eq:SeparationCondition}
\end{align}	
Then, the measure  $\sum_{w_k \in W}^s c_k \delta_{w_k}$ (and therefore also $f_0)$ can be recovered by solving the problem
	\begin{align*}
		\min_{\mu \in \calM(B_R(0))} \norm{\mu}_{TV} \st \sprod{K* \mu,m_n}=\sprod{f_0,M_n}, n =1, \dots N.
	\end{align*}
\end{theorem}

\begin{proof}
Let $\epsilon \in \left(0, \frac{1}{4}\right)$ be a parameter for which we will specify the value later.
Using the notation introduced above, consider the ansatz
	\begin{align*}
		g = \sum_{w_k \in W} \alpha_k \eta_{w_k} + \sum_{w_k \in W} \beta_k d\eta_{w_k}.
	\end{align*}
	Lemma \ref{lem:DerivativeVectors} implies that the conditions $g(w_k)= \sgn(c_k)$ and $g'(w_k)=0$ for each $k$ then corresponds to the system of equations
	\begin{align}
		\begin{bmatrix}
			\left(\sprod{\eta_{w_j}, \eta_{w_k}}\right)_{j,k=1}^s  & \left(\sprod{\eta_{w_j}, d\eta_{w_k}}\right)_{k,j=1}^s \\
			\left(\sprod{d\eta_{w_j}, d\eta_{w_k}}\right)_{k,j=1}^s & \left(\sprod{d\eta_{w_j}, d\eta_{w_k}}\right)_{k,j=1}^s
		\end{bmatrix} \cdot \begin{bmatrix}
		\alpha \\ \beta
		\end{bmatrix} = \begin{bmatrix}
		 (\sgn c_k)_{k=1}^s \\ 0
		\end{bmatrix}. \label{eq:CondAtPoints}
	\end{align}
	Let us define the matrix in the above equation as $\Phi(W)$.
Noticing that $\sprod{\eta_w,\eta_w}=\sprod{d\eta_w,d\eta_w}=1$, and that the off-diagonal entries are under control, we see that  $\Phi(W)$ is close to the identity.
Concretely, we have, using the explicit forms of the functions $\eta_w$ and $d\eta_w$,
\begin{align*} \norm{\Phi(W) - \id_{2s}}_\infty &\leq \sup_{i} \sup_{k=0,1} \sum_{j \neq i} \abs{\sprod{d^{(k)}K_{w_j},K_{w_i}}} + \abs{\sprod{d^{(k)}K_{w_j},dK_{w_i}}} \\
	 &\leq \sup_{i} \sum_{j \neq i} e^{-\frac{\abs{w_i- w_j}^2}{2}}(1 + \abs{w_i-w_j} +\abs{w_i-w_j}^2) \leq \overline{\sigma}_{\downarrow}(W) -1 \leq \epsilon,
	 \end{align*}
	 where we invoked the assumption that $\overline{\sigma}_{\downarrow}(W) -1$ is small. Now by a Neumann series argument (compare \cite{candes2014super}), we see that there will exist $\alpha$ and $\beta$ for which \eqref{eq:interpolate} is satisfied, and additionally,
	 \begin{align}
	 \abs{\alpha_k -1 }, \abs{\beta_k} \leq \frac{2\epsilon}{1-2\epsilon} \label{eq:alphabeta}
	 \end{align} for each $k$. \newline

Now we will check if the choice of the function $g$ which is induced by the above choice of $\alpha$ and $\beta$  satisfies $\abs{g}^2 <1$ for all $z \notin  W$. We will proceed in two steps: First, we will prove that for $z$ with $\dist(z,W) \leq \widetilde{\delta}$  $H_z(\abs{g}^2)$ will be negative definite. Then we will prove that for $z$ with $\dist(z,W) \geq \widetilde{\delta}$, $\abs{\sprod{g,K_z}}  <1$. $\widetilde{\delta}$ is thereby another parameter whose value we will specify later.

Let us begin by investigating $H_z(\abs{g}^2)$ for $z$ close to a $w_j$.  Towards this goal, let us first note that $\sprod{K_{w_j},g} = \sgn c_j$, $\sprod{dK_{w_j},g}=0$ implies that
\begin{align}
				\abs{\sprod{d\eta_z, g}}^2 - \abs{\sprod{\eta_z,g}}^2 =&
				\abs{\sprod{d\eta_z-d\eta_{w_j}, g}}^2 - \abs{\sprod{\eta_z-\eta_{w_j},g}}^2 \nonumber\\
				& - 2\re\left(\overline{\sgn c_j} \sprod{\eta_z-\eta_{w_j},g}\right) -1 \label{eq:Hessian1}\\
				 \overline{\sprod{\eta_z,g}}\sprod{d^2\eta_z,g} &= \overline{\sgn c_j} \sprod{d^2\eta_z, g} + \overline{ \sprod{\eta_z - \eta_{w_j},g}} \sprod{d^2\eta_z,g}. \label{eq:Hessian2}
\end{align}
Hence, if we can prove that $\abs{\sprod{d^{(k)}\eta_z-d^{(k)}\eta_{w_j}, g}}$ for $k=0,1$ as well as $\sprod{d^2\eta_z, g}$ is close to zero for $z$ close to $w_j$, Lemma \ref{lem:DerivativeVectors} proves that $dH_z$ is close to $-\id_{2s}$ for $z$ close to $w_j$, and hence negative definite.

Let us first note that for $w,z$ arbitrary and $\ell=0,1,2$, we have
\begin{align*}
	\abss{\sprod{d^{(l)}\eta_w-d^{(l)}\eta_z,g}} \leq \sum_{j=1}^s \abs{\alpha_j} \abs{\sprod{d^{(l)}\eta_w-d^{(l)}\eta_z, \eta_{w_j}}}
	+\sum_{j=1}^s \abs{\beta_j} \abs{\sprod{d^{(l)}\eta_w-d^{(l)}\eta_z, d\eta_{w_j}}}.
\end{align*}
Remembering \eqref{eq:alphabeta}, we obtain
\begin{align*}
	\abss{\abss{\sprod{d^{(l)}\eta_w-d^{(l)}\eta_z,g}}} \leq
	& \frac{1}{1-2\epsilon} \sum_{j=1}^s\abs{\sprod{d^{(l)}\eta_w-d^{(l)}\eta_z, \eta_{w_j}}}
 \\
 &+\sum_{j=1}^s \frac{2\epsilon}{1-2\epsilon} \abs{\sprod{d^{(l)}\eta_w-d^{(l)}\eta_z, d\eta_{w_j}}} + \epsilon,
\end{align*}
where it should not be forgotten that we have made $R$-related terms small by arguing that $N$ is large. By a continuity argument, we can further bound $\frac{1}{1-2\epsilon}\sum_{j=1}^s\abs{\sprod{d^{(l)}\eta_w-d^{(l)}\eta_z, d\eta_{w_j}}}$ uniformally for $z,w \in \C$ by a constant $K$ .

It is now only left to estimate $\sprod{d^{(\ell)}\eta_w- d^{(\ell)}\eta_z, \eta_{w_j}}$. (Note that due to $\sprod{d^{2}\eta_{w_j},\eta_{w_j}} = (\overline{z}-\overline{w})^2e^{-\frac{\abs{w}^2+\abs{z}^2}{2}} e^{\overline{w}z}\vert_{z=w_j} =0$, this also produces a bound for $\abs{\sprod{d^2K_w,K_{w_j}}}$).
This is however not hard: The mean value theorem implies $$
	\abs{\sprod{d^{(\ell)}\eta_w- d^{(\ell)}\eta_z, \eta_{w_j}}}   \leq \sup_{\xi \in [z,w]} \abs{\eta_{w_j}^{(\ell+1)}(\xi)}\abs{z-w},
$$
so that
\begin{align*}
	\abs{\sprod{d^{(\ell)}\eta_w- d^{(\ell)}\eta_z, \eta_{w_j}}} &\leq \left(\begin{cases} \sup_{\xi \in [z,w]}e^{-\frac{\abs{\xi-w_j}^2}{2}} \abs{\xi-w_j}, \quad \ell=0\\
	 \sup_{\xi \in [z,w]}e^{-\frac{\abs{\xi-w_j}^2}{2}}(1 + \abs{\xi-w_j}^2), \quad \ell=1 \\
	\sup_{\xi \in [z,w]}e^{-\frac{\abs{\xi-w_j}^2}{2}}\abs{\xi-w_j}(1 + \abs{\xi-w_j}^2), \quad \ell=2 \end{cases} \right) \cdot \abs{z-w} \\
\end{align*}
Summing up, and letting $\xi$ run through the entirety of $\C$, we obtain
\begin{align*}
	\abss{\abss{\sprod{d^{(l)}\eta_w-d^{(l)}\eta_z,g}}} &\leq \frac{\overline{\sigma}_\uparrow(W)}{(1-2\epsilon) }   \abs{z-w}  + (2K+1)\epsilon, \\
	\Longrightarrow \sup_{\abs{z-w_j} \leq \delta} \norm{H_z(\abs{g}^2) -(- \id_{2})} &\leq  \frac{2 \overline{\sigma}_\uparrow(W)}{1-2\epsilon}   \delta  +2\left(\frac{ \overline{\sigma}_\uparrow(W)}{1-2\epsilon}   \delta \right)^2  + C\overline{\sigma}_{\uparrow}^2(W)\epsilon
\end{align*}
where applied Lemma \ref{lem:DerivativeVectors},\eqref{eq:Hessian1}-\eqref{eq:Hessian2}, collected a number of constants to form $C$, and used the evident bound $\overline{\sigma}_{\uparrow}(W) \geq \overline{\sigma}(W,w_i) \geq 1$.
Here, the norm can be chosen to be the spectral norm: The same estimate for the  $\ell_1$ and $\ell_\infty$-operator norms  can be proven with the same technique used as for bounding $\id_{2s}-\Phi(W)$. Then,  an interpolation argument does the job.  If we now choose
$$\overline{\delta} = \frac{1}{\overline{\sigma}_{\uparrow}} \left( \sqrt{\frac{3}{4}- \frac{C\overline{\sigma}_{\uparrow}^2\epsilon}{2}} - \frac{1}{2}\right)(1-2\epsilon),$$  $H_z(\abs{g}^2)$ will be negative definite for each $z$ with $\dist(z,W)\leq\overline{\delta}$. Also  note that for $\epsilon$ small enough, we then have $\overline{\delta} \geq \frac{\sqrt{3}-1}{4\overline{\sigma}(W)_{\uparrow}}$.

It is by now easily estimate $\abs{\sprod{K_z,g}}$ for $z$ with $\dist(z,W)\geq \widetilde{\delta}$ directly. We have,
\begin{align*}
	\abs{\sprod{\eta_z,g}} &\leq \sum_{j=1}^s \abs{\alpha_j} \abs{\sprod{\eta_z, \eta_{w_j}}}  + \abs{\beta_j} \abs{\sprod{K_z, dK_{w_j}}}  \\
	&\leq \frac{1}{1-2\epsilon} \sum_{j=1}^s e^{-\frac{\abs{z-w_i}^2}{2}}  + \frac{2\epsilon}{1-2\epsilon} \sum_{j=1}^s e^{-\frac{\abs{z-w_i}^2}{2}} \abs{z-w_i}
\end{align*}
where we again invoked \eqref{eq:alphabeta}. Now we take the supremum over all $z$ with $\dist(z,W) \geq \overline{\delta}$. These automatically obey the estimate $\dist(z,W) \geq \frac{\sqrt{3}-1}{4\overline{\sigma}(W)_{\uparrow}}$, so that
\begin{align*}
	\abs{\sprod{\eta_z,g}} \leq \frac{1}{1-2\epsilon}(1-\tau) + \frac{2\epsilon}{1-2\epsilon} \overline{\sigma}_{\uparrow}(W) <1,
\end{align*}
for $\epsilon>0$ small enough. The proof is finished.

\end{proof}

\begin{remark}
In practice, we of course have to subsample $ A_\Lambda r$ somehow. We can for instance
do it as follows: Let $(M_n)_{n\in \N}:=(e^{-|z|^2/2}\tfrac{1}{\sqrt{n!}} z^n)_{n\in \N}$ denote the orthonormal basis of monomials for $\calH$. We then
measure
\begin{align*}
	\sprod{f, M_n}, \quad n=0 \dots N+1.
\end{align*}
Note that we can only need to have access to $r$ in order to calculate these values: $\sprod{f,M_n} = \sprod{A_\Lambda r, e_n} = \sprod{r, A_\Lambda^*e_n}$.

The range of these measurements is given by $\mathrm{span}(M_n)_{n=0}^{N+1}$. By using the following approximations of $\eta_w$ and $d\eta_w$:
 \begin{align*}
		p_{w_k}^N = e^{-\frac{\abs{w_k}^2+\abs{z}^2}{2}}\sum_{n=0}^{N} \frac{\overline{w_k}^n}{\sqrt{n!}} m_n , \ q_{w_k}^N = (z-w_k) e^{-\frac{\abs{w_k}^2+\abs{z}^2}{2}}\sum_{n=0}^{N} \frac{\overline{w_k}^n}{\sqrt{n!}}  m_n.
	\end{align*}
 instead of $\eta_{w_k}$ and $d\eta_{w_k}$ for our ansatz, we can apply a similar argument to above to deduce that if $N$ is large enough, exact recovery from these measurements. Due to the similarity to the above argument, we choose to not include the details here. These will instead be included in the upcoming PhD thesis of one of the authors.
\end{remark}

\subsection*{Acknowledgement}

B. G. Bodmann's research was partially supported by NSF grant DMS-1715735.
A. Flinth acknowledges support from the Berlin Mathematical School, as well as  DFG Grant KU 1446/18-1.
G. Kutyniok acknowledges partial support by the Einstein Foundation Berlin, the DFG Collaborative Research 
Center TRR 109 ``Discretization in Geometry and Dynamics'' and the European Commission Project DEDALE (contract 
no. 665044) within the H2020 Framework Program, and DFG-SPP 1798 Grants KU 1446/21 and KU 1446/23. She also
thanks the Department of Mathematics at the University of Houston for its hospitality and support during her
visit.

\end{document}